\documentclass[a4,12pt]{amsart}
\usepackage[utf8]{inputenc}
\usepackage{amsfonts,bm}
\usepackage{latexsym}
\usepackage{amssymb}
\usepackage{amsmath}

\usepackage[dvipsnames]{xcolor}

\usepackage{mathpazo}
\usepackage{domitian}
\usepackage[T1]{fontenc}

\usepackage[left=1.88cm, top=1.48cm,bottom=1.48cm,right=1.88cm]{geometry}

\usepackage{enumerate}
\usepackage{mathrsfs}
\usepackage{amsthm}
\usepackage{adjustbox}
\usepackage{comment} 

\usepackage{hyperref} 
\hypersetup{
    colorlinks,
    linkcolor=blue,
    citecolor=blue,
    urlcolor=blue
}

\definecolor{DarkOlive}{rgb}{0.25,0.30,0.00}

\definecolor{dog}{rgb}{0.30,0.40,0.15}

\usepackage{graphicx}

\usepackage{caption}
\usepackage{subcaption}
\numberwithin{equation}{section}

\usepackage[leftcaption]{sidecap}
\sidecaptionvpos{figure}{m}

\usepackage{subcaption}
\usepackage[bottom]{footmisc}

\usepackage{accents}

\definecolor{col1}{rgb}{0.6, 0.7, 0.8}
\definecolor{col2}{rgb}{0.7, 0.8, 0.65}
\definecolor{col3}{rgb}{0.8, 0.9, 0.5}
\definecolor{col4}{rgb}{0.91,0.94, 0.53}
\definecolor{col5}{rgb}{0.98,0.99,0.6}

\definecolor{c4}{rgb}{0.58, 0.69, 0.62}
\definecolor{c3}{rgb}{0.64, 0.76, 0.68}
\definecolor{c2}{rgb}{0.74, 0.86, 0.78}
\definecolor{c1}{rgb}{0.84, 0.96, 0.88}

\definecolor{bondiblue}{rgb}{0.0, 0.58, 0.71}

\definecolor{d1}{rgb}{0.6, 0.6, 0.6}

\definecolor{d2}{rgb}{0.35, 0.68, 0.91}
\definecolor{d3}{rgb}{0.43, 0.62, 0.89}
\definecolor{d4}{rgb}{0.43, 0.62, 0.89}
\definecolor{d5}{rgb}{0.49, 0.67, 0.84}
\definecolor{d6}{rgb}{0.43, 0.62, 0.89}
\definecolor{d7}{rgb}{0.35, 0.68, 0.91}
\definecolor{d8}{rgb}{0.43, 0.62, 0.89}

\definecolor{d2}{rgb}{0.3, 0.68, 0.8}

\definecolor{pd}{rgb}{0.8,0.1,0.1}

\definecolor{DarkOrchidHalf}{rgb}{0.3,0.098,0.4}

\definecolor{textcol}{rgb}{0.37,0,0.57}
\usepackage{tikz}
\usepackage{forest}
\forestset{
  default preamble={
   for tree={circle, draw, 
            minimum size=2.1em, 
            inner sep=1pt} 
  }
}

\usetikzlibrary{positioning}

\newtheorem{thm}{Theorem}[section]

\newtheorem{cor}[thm]{Corollary}
\newtheorem{lemma}[thm]{Lemma}

\newtheorem{proposition}[thm]{Proposition}

\newtheorem{prob}[thm]{Problem}

\usepackage{nomencl}
\makenomenclature

\title{Free probability via entropic optimal transport}
\author{Octavio Arizmendi}
\address{CIMAT, 36023 Guanajuato, Gto., M\'exico}
\email{octavius@cimat.mx}

\author{Samuel G. G. Johnston}
\address{King's College London, London WC2R 2LS, United Kingdom}
\email{samuel.g.johnston@kcl.ac.uk}


\begin{document}

\maketitle

\begin{abstract}
Let $\mu$ and $\nu$ be compactly supported probability measures on $\mathbb{R}$, and let $\mu \boxplus \nu$ denote their additive free convolution.
We show that for all sufficiently large real $z$,
\begin{equation} \label{eq:inter}
\int_{-\infty}^\infty \log(z - x) \, (\mu \boxplus \nu)(\mathrm{d}x)
= \sup_{\Pi} \left\{ \mathbf{E}_\Pi[\log(z - (X+Y))] - H(\Pi \mid \mu \otimes \nu) \right\},
\end{equation}
where the supremum is taken over all couplings $\Pi$ of $\mu$ and $\nu$.
Analogous formulas hold for multiplicative free convolution $\mu \boxtimes \nu$ and free compression $[\mu]_\tau$.
In this way, integrals of a log-potential against free convolutions can be expressed as entropic optimal transport problems.
The corresponding optimal couplings admit explicit formulas, from which the standard $R$- and $S$-transform descriptions of additive and multiplicative free convolution are recovered.
In particular, the optimizer $\Pi_z$ in \eqref{eq:inter} encodes the subordination equations via couplings of random variables.
Our approach is based on a large deviation principle on the symmetric group, combined with the quadrature method of Marcus--Spielman--Srivastava.

\end{abstract}


\section{Introduction and overview}

\subsection{The basic operations of free probability}
In the late 1980s and 1990s, Voiculescu pioneered the field of free probability \cite{Vcon1, Vcon2, Vcon3, Vent1, Vent2, Vent3}. From the concrete perspective of random matrix theory, free probability describes how the spectra of large independent random matrices interact under the basic matrix operations: matrix addition, matrix multiplication, and taking matrix minors. 

With a view to outlining some of the central results of free probability here, we say an $N \times N$ random matrix $A_N$ is unitarily invariant if $A_N$ has the same law as $U_NA_NU_N^{-1}$, where $U_N$ is a Haar distributed unitary random matrix. 
In the following, consider sequences $(A_N)_{N \geq 1}$ and $(B_N)_{N \geq 1}$ of independent unitarily invariant random Hermitian matrices. 
Suppose further that as $N \to \infty$ the empirical spectra of these sequences converge (in the weak topology) to probability measures $\mu$ and $\nu$ on the real line in the sense that 
\begin{equation} \label{eq:conv00}
\frac{1}{N} \sum_{i=1}^N \delta_{\lambda_i(A_N)} \to \mu  \qquad \text{and} \qquad \frac{1}{N} \sum_{i=1}^N \delta_{\lambda_i(B_N)} \to \nu,
\end{equation}
where for $i=1,\ldots,N$, $\lambda_i(C_N)$ denote the eigenvalues of an $N \times N$ matrix $C_N$.
Free probability grants descriptions of the asymptotic empirical spectra of sums, products, and minors of sequences of random matrices $(A_N)_{N \geq 1}$ and $(B_N)_{N \geq 1}$ satisfying \eqref{eq:conv00}.

The first free probability operation we discuss, \textbf{additive free convolution}, takes two probability measures $\mu$ and $\nu$ and outputs a new probability measure $\mu \boxplus \nu$ on the real line. This new probability measure has the property that, under \eqref{eq:conv00}, the (random) empirical spectrum of $A_N + B_N$ converges weakly almost surely to $\mu \boxplus \nu$, i.e.\ 
\begin{align} \label{eq:aconv}
\frac{1}{N} \sum_{ i = 1}^N \delta_{\lambda_i( A_N + B_N ) } \to  \mu \boxplus \nu \qquad \text{$A_N,B_N$ indep., unitarily inv., and satisfy \eqref{eq:conv00}.}
\end{align} 
The second operation, \textbf{multiplicative free convolution}, describes the asymptotic empirical spectrum of the product of random matrices. 
Here we find that in parallel to \eqref{eq:aconv}, provided $\nu$ is supported on $[0,\infty)$ we have the almost-sure weak convergence
\begin{align} \label{eq:mconv}
\frac{1}{N} \sum_{ i = 1}^N \delta_{\lambda_i(A_NB_N)} \to  \mu \boxtimes \nu \qquad \text{$A_N,B_N$ indep., unitarily inv., and satisfy \eqref{eq:conv00}.}
\end{align} 
The final operation, \textbf{free compression} $[\mu]_\tau$, describes the asymptotic spectra of minors of random matrices. Namely, given an $N \times N$ matrix $A_N$, write $[A_N]_k$ for the principal $k \times k$ minor (i.e.\ top-left corner) of $A_N$. Then for $\tau \in (0,1]$, there is a probability measure $[\mu]_\tau$ such that we have the almost-sure weak convergence
\begin{align} \label{eq:cconv}
\frac{1}{k} \sum_{ i = 1}^k \delta_{\lambda_i([A_N]_k)} \to  [\mu]_\tau \qquad \text{$k = k_N = \lfloor \tau N\rfloor$, $A_N$ unitarily inv., and satisfies \eqref{eq:conv00}},
\end{align} 
where $\lfloor \tau N\rfloor$ refers to the greatest integer less than $\tau N$. Our notation $[\mu]_\tau$ here is unconventional. Usually a different notation is used based on a formulation of free compression in terms of additive free convolution; see Section \ref{sec:compr}.

\subsection{Characterisations of the free probability operations}
The three fundamental operations of free probability, $\mu \boxplus \nu$, $\mu \boxtimes \nu$, and $[\mu]_\tau$, are somewhat difficult to describe. Take for instance the additive free convolution $\mu \boxplus \nu$ of two probability measures. Perhaps the most direct way to describe additive free convolution is as follows. Given a probability measure $\pi$ on $\mathbb{R}$ with compact support, define the Cauchy transform $G_\pi: \mathbb{C} - \mathrm{supp}(\pi) \to \mathbb{C}$ of $\pi$ by 
\begin{align} \label{eq:steel}
G_\pi(z) := \int_{-\infty}^\infty \frac{1}{z-x} \pi(\mathrm{d}x).
\end{align}
Now define the $R$-transform of the measure $\pi$ by setting
\begin{align*}
R_\pi(s) := - \frac{1}{s} + G_\pi^{-1}(s),
\end{align*}
for sufficiently small $s$. Here $G_\pi^{-1}$ is the (locally defined) inverse function of $G_\pi$. The additive free convolution $\mu \boxplus \nu$ is then the probability measure on the real line satisfying
\begin{align} \label{eq:Rtran}
R_{\mu \boxplus \nu}(s) = R_\mu(s) + R_\nu(s).
\end{align}

There is an analogous (albeit slightly more complicated) description of the multiplicative free convolution $\mu \boxtimes \nu$, and for free compression $[\mu]_\tau$, there is a description in terms of the additive free convolution.

Alternatively, the basic operations of free probability can also be described in combinatorial terms \cite{NS}, with the notion of free cumulants playing an analogous role to that of classical cumulants in classical probability. 
Nonetheless, these descriptions of the basic operations of free probability are somewhat indirect. 

\subsection{Entropic optimal transport}
The main result of the present article offers a new characterisation of $\mu \boxplus \nu$, $\mu \boxtimes \nu$ and $[\mu]_\tau$ in terms of \textbf{entropic optimal transport} \cite{LMS, NW, BGN, nutz}.

We now run through the very basics of optimal transport \cite{villani}. Recall that we say a probability measure $\Pi$ on $\mathbb{R}^2$ is a \textbf{coupling} of two probability measures $\mu$ and $\nu$ on $\mathbb{R}$ if its pushforward measures under projections onto the coordinate axes are given by $\mu$ and $\nu$ respectively. Namely,
\begin{align*}
\Pi ( E \times \mathbb{R} ) = \mu(E) \qquad \text{and} \qquad \Pi( \mathbb{R} \times F ) = \nu(F)
\end{align*}
for all Borel subsets $E$ and $F$ of $\mathbb{R}$. The simplest example of a coupling is product measure $\mu \otimes \nu$, i.e.\ $(\mu \otimes \nu)( E \times F ) = \mu(E) \nu(F)$. Given a coupling $\Pi$, we write $(X,Y)$ for the $\mathbb{R}^2$-valued random variable associated with this coupling.

Given a measurable function $c:\mathbb{R}^2 \to \mathbb{R}$, the classical Monge-Kantorovich problem is to solve the minimisation and maximisation problems
\begin{align*}
\inf_{ \Pi } \mathbf{E}_\Pi[ c(X,Y) ] \qquad \text{and} \qquad \sup_{ \Pi} \mathbf{E}_\Pi[c(X,Y)],
\end{align*}
where the infimum and supremum run over all couplings of a given pair of probability measures $\mu$ and $\nu$. 
Perhaps the most natural example is taking $c(x,y) = |y-x|$ in the inf setting; in this case the problem represents the distance-minimising way of transporting the mass of a measure $\mu$ to that of $\nu$. Another example in the infimum setting is $c(x,y) = \mathrm{1}_{y\neq x}$; with this example the infimum coincides
with the total variation distance between the measures $\mu$ and $\nu$.  

In \emph{entropic} optimal transport, one introduces an entropy cost to the Monge-Kantorovich problem. 
To this end, given two probability measures $\Pi$ and $\Pi'$ on a measurable space $E$, if $\Pi'$ is absolutely continuous with respect to $\Pi$, define the \textbf{relative entropy} (also known as the Kullback-Leibler divergence) by
\begin{align} \label{eq:kl}
H(\Pi'|\Pi) := \int_E  \log \frac{ \mathrm{d}\Pi'}{\mathrm{d}\Pi} ~\mathrm{d}\Pi'.
\end{align}
If $\Pi'$ is not absolutely continuous with respect to $\Pi$, set $H(\Pi'|\Pi)=+\infty$. We have $H(\Pi'|\Pi) \geq 0$, with equality if and only if $\Pi' = \Pi$. 

Given a measurable function $c:\mathbb{R}^2 \to \mathbb{R}$ and a parameter $\kappa > 0$, the entropic optimal transport problem is then to solve the maximisation problem
\begin{align*}
\sup_\Pi \left\{ \mathbf{E}_\Pi [c(X,Y)] - \kappa H(\Pi|\mu \otimes \nu) \right\},
\end{align*}
where again the supremum runs over all couplings $\Pi$ of probability measures $\mu$ and $\nu$.

\subsection{Main result}
Given a measure $\mu$ on $\mathbb{R}$, let $[E_\mu^-,E_\mu^+]$, with $-\infty \leq E_\mu^- \leq E_\mu^+ \leq +\infty$, denote the smallest closed interval containing the support of $\mu$. Thus $\mu$ has compact support if and only if $-\infty < E_\mu^- \leq E_\mu^+ < \infty$. In the setting of multiplicative free convolution below, it will be useful to observe that if $\nu$ is supported on $[0,\infty)$, then a random variable $XY$ with $X \sim \mu$ and $Y \sim \nu$ (under any coupling) does not exceed $\max \{ E_\mu^+ E_\nu^+, E_\mu^- E_\nu^-\}$ almost surely.

Our main result states that the integral of a log-potential against each of the measures $\mu \boxplus \nu$, $\mu \boxtimes \nu$, and $[\mu]_\tau$ can be expressed in terms of an entropic optimal transport problem: 

\begin{thm}\label{thm:main}
Let $\mu$ and $\nu$ be probability measures on the real line with compact support. Then for $z > E_\mu^+ + E_\nu^+$ we have 
\begin{align} \label{eq:max1}
\int_{-\infty}^\infty \log(z-x) ( \mu \boxplus \nu )(\mathrm{d}x) =\sup_{\Pi} \left\{ \mathbf{E}_\Pi [ \log (z - (X+Y)) ] - H(\Pi|\mu \otimes \nu) \right\}.
\end{align}
If we additionally suppose that $\nu$ is supported on $[0,\infty)$, then whenever  $z > \max \{ E_\mu^+ E_\nu^+, E_\mu^- E_\nu^-\} $ we have 
\begin{align} \label{eq:max2}
\int_{-\infty}^\infty \log(z-x)  (\mu \boxtimes \nu) (\mathrm{d}x)  =\sup_{\Pi} \left\{ \mathbf{E}_\Pi [ \log (z - XY) ] -H(\Pi|\mu \otimes \nu) \right\}.
\end{align}
In both \eqref{eq:max1} and \eqref{eq:max2}, the supremum is taken over all couplings $\Pi$ of $\mu$ and $\nu$.

Finally, for $z > E_\mu^+$ we have 
\begin{align} \label{eq:max3}
\tau \int_{-\infty}^\infty  \log(z-x)  [\mu]_\tau (\mathrm{d}x) =\sup_{\Pi} \left\{ \mathbf{E}_\Pi [ \log (z - X)\mathrm{1}_{Y > 1-\tau} ] - H(\Pi|\mu \otimes \nu_{\mathrm{uni}}) \right\},
\end{align}
where $\nu_{\mathrm{uni}}$ refers to the Lebesgue measure on $[0,1]$, and the supremum is taken over all couplings $\Pi$ of $\mu$ and $\nu_{\mathrm{uni}}$.
\end{thm}

We will see in the next section that in each of the contexts \eqref{eq:max1}, \eqref{eq:max2} and \eqref{eq:max3}, there is a unique coupling $\Pi_{*,z}$ achieving the supremum. 
We will refer to this coupling as the \textbf{optimal coupling}, and emphasise that it depends on the real number $z$. 

\subsection{Explicit solutions to the maximisation problems} \label{sec:explicit}
In this section we outline how one can solve the maximisation problems on the right-hand sides of \eqref{eq:max1}, \eqref{eq:max2} and \eqref{eq:max3} explicitly. For the sake of brevity, we will focus here on the former of these problems, \eqref{eq:max1}; the solutions to \eqref{eq:max2} and \eqref{eq:max3} are similar.

We begin with the following result from entropic optimal transport.

\begin{lemma}[Proposition 2.2 of Bernton, Ghosal and Nutz \cite{BGN}] \label{lem:bgn}
Let $c:\mathbb{R}^2 \to \mathbb{R}$ be measurable and bounded. Then there is a unique optimal coupling $\Pi_{*}$ of probability measures $\mu$ and $\nu$ maximising the functional
\begin{align*}
\mathcal{F}_c[\Pi] := \mathbf{E}_\Pi [ c(X,Y) ] - H(\Pi|\mu \otimes \nu).
\end{align*}
Moreover, the optimal coupling $\Pi_*$ is the unique coupling of $\mu$ and $\nu$ with Radon-Nikodym derivative against product measure taking the form
\begin{align} \label{eq:xca}
\frac{ \mathrm{d} \Pi_{*}}{\mathrm{d}(\mu \otimes \nu) } = A(X)B(Y)e^{c(X,Y)}
\end{align}
for some measurable functions $A,B:\mathbb{R} \to [0,\infty)$. 

\end{lemma}


A brief calculation using the definition \eqref{eq:kl} of the relative entropy $H(\Pi | \mu \otimes \nu)$ tells us that if $\Pi_*$ is the optimal coupling given in \eqref{eq:xca}, then 
\begin{align} \label{eq:ecom}
\sup_{\Pi} \mathcal{F}_c[\Pi] = \mathcal{F}_c[\Pi_{*}] = \mathbf{E}_{\Pi_*}[-\log A(X)-\log B(Y)] = - \mathbf{E}_{\mu}[ \log A(X) ] - \mathbf{E}_{\nu} [\log B(Y)],
\end{align}
where in the last equality above we have used the fact that $\Pi_*$ is a coupling of $\mu$ and $\nu$.


We now apply Lemma \ref{lem:bgn} to \eqref{eq:max1}.
Setting $c(x,y) = \log(z-(x+y))$ in Lemma \ref{lem:bgn}, the coupling maximising the right-hand side of \eqref{eq:max1} is the unique coupling of $\mu$ and $\nu$ whose Radon-Nikodym derivative against product measure takes the form $ \mathrm{d} \Pi_{*,z}/\mathrm{d}(\mu \otimes \nu) = A_z(X)B_z(Y)( z - (X+Y))$ for nonnegative measurable $A_z$ and $B_z$ depending on $z$. (Note the assumption above \eqref{eq:max1} that $z > E_\mu^+ + E_\nu^+$ guarantees $z - (X+Y)$ is positive.) After a series of manipulations (see Section \ref{sec:free} for details), one can check that this forces $\Pi_{*,z}$ to take the form 
\begin{align} \label{eq:xca2}
\frac{ \mathrm{d} \Pi_{*,z}}{\mathrm{d}(\mu \otimes \nu) } = \omega(z) \frac{ z - (X+Y)}{ (\omega_\mu(z) - X)(\omega_\nu(z) - Y) },
\end{align}
where $\omega(z),\omega_\mu(z),\omega_\nu(z)$ are the unique solutions to the subordination equations
\begin{align}
\omega(z) &= \omega_\mu(z) + \omega_\nu(z) \label{eq:add10} \\
 \frac{1}{\omega(z) - z } &= G_\mu(\omega_\mu(z)) = G_\nu(\omega_\nu(z)), \label{eq:add20}
\end{align}
and, as in \eqref{eq:steel}, $G_\pi$ is the Cauchy transform of a measure $\pi$. Using \eqref{eq:ecom} in conjunction with the fact that in this case we have $A_z(X)B_z(Y) :=  \omega(z)(\omega_\mu(z) - X)^{-1}(\omega_\nu(z) - Y)^{-1}$, we obtain that for all $z > E_\mu^+ + E_\nu^+$ we have 
\begin{align} \label{eq:apollo20}
\int_{-\infty}^\infty \log(z-x) \mu \boxplus \nu (\mathrm{d}x) = - \log\omega(z) +  \mathbf{E}_\mu[ \log(\omega_\mu(z) - X) ]  + \mathbf{E}_\nu[ \log(\omega_\nu(z) - Y) ]  ,
\end{align}
where again, $\omega(z),\omega_\mu(z),\omega_\nu(z)$ are the unique solutions to \eqref{eq:add10}-\eqref{eq:add20}. The equation \eqref{eq:apollo20} is a variation on well-known subordination results in the literature, see e.g.\ \cite{BB, BZ, Bia, Vent1, Voi02}. 

In Section \ref{sec:free} we show that one can differentiate \eqref{eq:apollo20} with respect to $z$ to 
characterise the Cauchy transform $G_{\mu \boxplus \nu}(z)$ of the free additive convolution, and thereby derive the $R$-transform relation \eqref{eq:Rtran}.

We can also solve explicitly the corresponding optimal transport problems \eqref{eq:max2} and \eqref{eq:max3} for multiplicative free convolution and free compression, and thereby obtain the corresponding subordination descriptions of these measures. In summary, we have the following result.

\begin{thm} \label{thm:freecons}
Each of the optimal transport problems on the right-hand sides of \eqref{eq:max1}, \eqref{eq:max2} and \eqref{eq:max3} may be solved explicitly. In the case of \eqref{eq:max1}, the solution leads to the $R$-transform description of additive free convolution. In the case of \eqref{eq:max2}, the solution leads to the $S$-transform description of multiplicative free convolution. In the case of \eqref{eq:max3}, the solution leads to the description of free compression in terms of rescaled additive free convolution. 
\end{thm}

See Section \ref{sec:variational} and Section \ref{sec:free} for more explicit statements of the various aspects of Theorem \ref{thm:freecons}.

\subsection{Proof overview} \label{sec:overview}
Our proof of Theorem \ref{thm:main} is based on an asymptotic analysis of the remarkable \textbf{quadrature formulas} developed by Marcus, Spielman and Srivastava in \cite{MSS4} and \cite{MSS3}. The quadrature formulas state that the expected characteristic polynomial of certain operations of sums, products and minors of random conjugations of matrices remains the same, regardless of whether the conjugation is done by a random unitary or random permutation matrix.

To give a brief flavour here we will outline the additive case, where an asymptotic analysis of the additive quadrature formula leads to \eqref{eq:max1}. (The approaches to proving \eqref{eq:max2} and \eqref{eq:max3} are similar.) The additive quadrature formula states that for real $N\times N$ diagonal matrices $(A_N)_{N \geq 1}$ and $(B_N)_{N \geq 1}$ we have the surprising relation
\begin{align} \label{eq:GM0}
\mathbf{E} [ \det (z - (A_N + U_NB_NU_N^{-1}) ) ]  = \mathbf{E} [ \det (z - (A_N + \Sigma_N B_N \Sigma_N^{-1} )) ],
\end{align}
where on the left-hand side $U_N$ is a Haar unitary random matrix, and on the right-hand side $\Sigma_N$ is a uniform random permutation matrix.
We study the large-$N$ asymptotics of \eqref{eq:GM0} as $A_N$ and $B_N$ are taken as sequences of diagonal matrices whose empirical spectra approximate probability measures $\mu$ and $\nu$ (i.e., we are in the setting of \eqref{eq:conv00}).

When one studies the asymptotics of the left-hand side of \eqref{eq:GM0} under \eqref{eq:conv00}, we find that this expectation is governed by the \emph{typical} behaviour of the empirical spectrum of $A_N + U_N B_N U_N^{-1}$, which, in light of \eqref{eq:aconv}, converges to $\mu \boxplus \nu$. Indeed, in Section \ref{sec:Uconv} we use recently developed tools from finite free probability to prove that under \eqref{eq:conv00} we have
\begin{align} \label{eq:UGM000}
\lim_{N \to \infty} \frac{1}{N } \log \mathbf{E} [ \det (z - (A_N + U_NB_NU_N^{-1}) ) ]   =  \int_{-\infty}^\infty \log(z-x) ( \mu \boxplus \nu )(\mathrm{d}x) . 
\end{align}

The asymptotics of the right-hand side of \eqref{eq:GM0} under \eqref{eq:conv00} are substantially more delicate. Here the expectation gets its overwhelming contribution from \emph{unusual} permutation matrices $\Sigma_N$ that pair up the eigenvalues of $A_N$ and $B_N$ in such a way as to make the determinant larger than its typical value. 
The asymptotic cost (in terms of probability) for a uniform permutation matrix to take on such an unusual value is captured by a large deviation principle on the symmetric group due to Kenyon et al.\ \cite{KKRW}, based on earlier work by Trashorras \cite{trashorras} and Wu \cite{wu}. 
Giving a very rough idea of how this works here, we find that given a coupling $\Pi$ of the laws $\mu$ and $\nu$, we have 
\begin{align} \label{eq:fern1}
\mathbf{P} \left( \text{Emp. spect. of $A_N + \Sigma_N B_N \Sigma_N^{-1}$ close to law of $X+Y$ under $\Pi$} \right) = e^{ - (1 + o(1)) N H(\Pi | \mu \otimes \nu) }.
\end{align}
We note in particular that since $H(\Pi|\mu \otimes \nu ) = 0$ only when $\Pi = \mu \otimes \nu$, \eqref{eq:fern1} entails that the empirical spectrum of $A_N+\Sigma_NB_N\Sigma_N^{-1}$ converges to $\mu \ast \nu$.

As $N \to \infty$, the right-hand side of \eqref{eq:GM0} is dominated by $\Sigma_N$ that attempt to maximise the determinant $\det(z - (A_N + \Sigma_N B_N \Sigma_N^{-1}))$ whilst trying to minimise the relative entropy $e^{ - N H(\Pi | \mu \otimes \nu) }$. 
Using Varadhan's lemma from large deviation theory, we ultimately prove under a strong version of \eqref{eq:conv00} that
\begin{align} \label{eq:SGM000}
\lim_{N \to \infty} \frac{1}{N } \log  \mathbf{E} [ \det (z - (A_N + \Sigma_N B_N \Sigma_N^{-1} )) ] = \sup_{\Pi} \left\{ \mathbf{E}_\Pi [ \log (z - (X+Y)) ] - H(\Pi|\mu \otimes \nu) \right\},
\end{align}
where the supremum is taken over all couplings $\Pi$ of $\mu$ and $\nu$.

In any case, comparing \eqref{eq:UGM000} and \eqref{eq:SGM000} using \eqref{eq:GM0}, we obtain the first equation \eqref{eq:max1} of our main result. As mentioned above, the proofs of \eqref{eq:max2} and \eqref{eq:max3} are similar in flavour. 

\subsection{Related work}

Free probability was initiated by Voiculescu in the 1980s and 1990s \cite{Vcon1, Vcon2, Vcon3, Vent1, Vent2, Vent3}.  Additive and multiplicative free convolution first appeared in the respective works \cite{Vcon1} and \cite{Vcon2} of Voiculescu. Free compression first appeared explicitly in Nica and Speicher \cite{NS}. 

Since the inception of free probability, numerous authors have noted connections with optimal transport.
Perhaps the first work in this vein was by Guionnet and Shlyakhtenko \cite{GS}, who proved a free analogue of Brenier's theorem.
Bahr and Boschert \cite{BaBo} recently gave a free analogue of a result by Cordero-Erausquin and Klartag \cite{CEK} on moment measures (see also \cite{sant}). 
Arizmendi et al.\ \cite{ACSY} provides another example of a variational principle in free probability.
Gangbo et al.\ prove, among other things, a free analogue of Monge-Kantorovich duality \cite{GJNS}. 
Jekel, Li and Shlyakhtenko construct a free analogue of the Wasserstein manifold in \cite{JLS}.
See also further work by Jekel \cite{jekel, jekel2}.

Entropic optimal transport is a very active subfield \cite{LMS, NW, BGN}. One of the principal motivations behind entropic optimal transport is that in general (and indeed, in many natural settings, for instance when $c(x,y)$ is Euclidean distance), the couplings $\Pi$ minimising $\mathbf{E}_\Pi [c(X,Y)]$ in Monge-Kantorovich problems are non-unique. However, as soon as an entropy penalty is introduced, the optimal couplings are unique. Under fairly mild conditions, one can send the entropy penalty to zero and obtain a `canonical' choice of coupling minimising $\mathbf{E}_\Pi [c(X,Y)]$. We refer the reader to lecture notes by Nutz for further information \cite{nutz}.

The quadrature formulas appear explicitly in \cite{MSS4, MSS3} by Marcus, Spielman, and Srivastava, and have their roots in the theory of interlacing families, which these authors developed in their work on Ramanujan graphs \cite{MSS} and in their resolution of the Kadison--Singer problem \cite{MSS2}.
It is noted in \cite{MSS3} that the idea of polynomial convolutions dates back to Walsh \cite{walsh} and Szeg\"o \cite{szego}.
See also the recent work of Gorin and Marcus \cite{GM} on quadrature formulas and the crystallization of random matrix orbits.

The main result of the present article involves classical entropy appearing in the context of free probability. It bears mentioning that in a series of papers, starting with \cite{Vent1} and \cite{Vent2}, Voiculescu introduced a free analogue of classical entropy. Free entropy describes, roughly speaking, the volume (in terms of Lebesgue measure on Hermitian matrices) of tuples of matrices whose joint tracial moments approximate those of certain non-commutative random variables \cite[Chapter 7]{MS}. For a single probability measure $\mu$, its free entropy is characterised by the formula
\begin{align} \label{freeentropy}
\chi(\mu) := \int_{-\infty}^\infty \int_{-\infty}^\infty \log|y-x|\mu(\mathrm{d}x) \mu(\mathrm{d}y) + \frac{1}{2} \log(2 \pi) + \frac{3}{4}. 
\end{align}
Shlyakhtenko and Tao \cite{ST} showed that the free entropy $\chi((\sqrt{k})_* [\mu]_{1/k})$ of the rescaled free compression is monotone increasing for real $k \geq 1$ (see also Shlyakhtenko \cite{Shlya} for the case where $k$ is an integer). 
In the same work, they provided a variational description of the flow of probability measures $( (\sqrt{k})_* [\mu]_{1/k} )_{k \geq 1}$, highlighting connections between the statistical physics of Gelfand--Tsetlin patterns, the free energy of the \emph{bead model} \cite{boutillier}, and free compression.
This variational description is further clarified in a recent preprint by the second author and Joscha Prochno \cite{JP}, where it is shown that random Gelfand--Tsetlin patterns satisfy a large deviation principle with a rate function generalising Voiculescu's free entropy. This work relies on explicit formulas for bead configurations on the torus developed in \cite{johnston}. See also \cite[Sections~3--4]{JO} for a related conjecture in the setting of the Toda lattice.

There are intimate relationships between the large deviations of random matrices and free probability that have been studied in work of Guionnet and coauthors \cite{AG, BGH, BCG, CG, CG2, GZ}; see in particular Guionnet's survey article \cite{guionnet}. Matrix concentration inequalities may be used to bound the error in the free probability approximations of spectral measures of sums and products of random matrices as the dimensions of the matrices are sent to infinity \cite{BBV}. Let us highlight in particular Narayanan and Sheffield \cite{NS2} (see also \cite{BGH, NST}), who recently showed that roughly speaking under \eqref{eq:conv00}, for probability measures $\pi$ on $\mathbb{R}$ we have
\begin{align} \label{eq:fern2}
\mathbf{P} \left( \text{Emp. spect. of $A_N + U_N B_N U_N^{-1}$ close to $\pi$} \right) \approx e^{ - (1 + o(1)) N^2 H_{\mu,\nu}[\pi] };
\end{align} c.f.\ \eqref{eq:fern1}. The functional $H_{\mu,\nu}[\pi]$ is nonnegative, and zero only when $\pi = \mu \boxplus \nu$ is additive free convolution. We note that the large deviation principle in \eqref{eq:fern2} has $e^{-O(N^2)}$ behaviour, whereas the determinant $\mathrm{det}(z-(A_N+U_NB_NU_N^{-1}))$ has $e^{O(N)}$ behaviour. It follows that unlike in the setting of \eqref{eq:SGM000}, where there is nontrivial interplay between the large deviation principle \eqref{eq:fern1} and the size of the determinant, in the setting of \eqref{eq:UGM000} the expectation is dominated by the event that the empirical spectrum of $A_N+U_NB_NU_N^{-1}$ is close to $\mu \boxplus \nu$. 

Not much is known explicitly about the functional $H_{\mu,\nu}[\pi]$ of a real probability measure $\pi$ occurring in \eqref{eq:fern2}, but it can be described in terms of a (non-explicit) free energy for hives, a combinatorial object introduced by Knutson and Tao \cite{KT,KT2} that describes the possible eigenvalues of the sum of two Hermitian matrices. Where the macroscopic behaviour of random Gelfand-Tsetlin patterns leads to a variational understanding of free compressions \cite{JP}, an outstanding open problem is to better understand the macroscopic behaviour of random hives \cite{taoblog}, and thereby obtain an explicit variational understanding of additive free convolution.


In Section \ref{sec:setup} we discuss \emph{permutons}, which are continuum limits of permutations introduced in \cite{HKMRS} and \cite{GGK}. In addition to the large deviation principle for uniform permutations that we use \cite{KKRW}, numerous authors have studied in recent years the macroscopics of random permutations from other probability measures on the symmetric group (such as Mallows permutations). See e.g.\ \cite{ADK, dauvergne, RVV}. 

\subsection{Overview of article}
The remainder of the paper is organised as follows. 
\begin{itemize}
\item Section \ref{sec:further} is devoted to corollaries, generalisations, and further discussion of the main results.
In particular, we derive inequalities relating free and classical operations, and give a multivariate version of the main result. 
We describe a probabilistic interpretation of the subordination equations and outline its application to relating the limit theory of iterated free and classical convolutions.
\item In Section \ref{sec:3} we introduce the quadrature formulas in full, and discuss couplings, copulas and relative entropy. Here we set up the proof of Theorem \ref{thm:main} and its multivariate version, which is then completed in the following two sections.
\item In Section \ref{sec:Uconv}, we study asymptotics of expectations of characteristic polynomials under unitary conjugation, and prove \eqref{eq:UGM000}, as well as its associated analogues for multiplication and minors of matrices.
\item In Section \ref{sec:Sconv}, we study asymptotics of expectations of characteristic polynomials under symmetric group conjugation, proving \eqref{eq:SGM000} and its analogues for multiplication and minors of matrices.
\item In Section \ref{sec:variational}, we discuss further the explicit solutions to the variational problems set out in \eqref{eq:max1}, \eqref{eq:max2} and \eqref{eq:max3}. 
\item In the final Section \ref{sec:free}, we expand on Theorem \ref{thm:freecons}, thereby connecting the solutions of the variational problems with standard results in the free probability literature. 
\end{itemize}

\subsection{A word on notation}
Throughout the article, we are able to treat most equations involving additive and multiplicative free convolution (and the corresponding addition and multiplication of random matrices) simultaneously, and do this by introducing a placeholder $\odot$ that represents either addition or multiplication of matrices or of real numbers. Namely,
\begin{align*}
\odot := + \qquad \text{or} \qquad \odot := \times
\end{align*}
We use the corresponding notation $\boxdot$ to denote the corresponding additive or multiplicative free convolution, i.e.\ $\boxdot = \boxplus$ or $\boxdot = \boxtimes$. The notation $\otimes$ always refers to product measure.

Recall that we write $[E_\mu^-,E_\mu^+]$ for the smallest closed interval containing the support of a probability measure $\mu$. 
On occasion we will use the notation $\max \prod_{j=1}^d E_{\mu_j}^{\pm}$ to denote 
\begin{align*}
\max_{\varepsilon \in \{-,+\}^d } \prod_{j=1}^d E_{\mu_j}^{\varepsilon_j},
\end{align*}
where the maximum is taken over all $2^d$ choices $\varepsilon = (\varepsilon_1,\ldots,\varepsilon_d)$ with each $\varepsilon_j$ equal to either $-$ or $+$. 
This coincides with the supremum of $X_1 \cdots  X_d$ taken over all couplings of random variables $X_1,\ldots,X_d$ with marginal laws $\mu_1,\ldots,\mu_d$.

\subsection*{Acknowledgements}
The authors would like to thank Dimitri Shlyakhtenko, Jan Nagel, Peter M{\"o}rters, Sam Power, Zakhar Kabluchko and Otte Heinävaara for their valuable comments. 

\section{Corollaries, generalisations, and further discussion} \label{sec:further}

\subsection{Free-classical inequalities}
We begin with a quick application of our main result, Theorem \ref{thm:main}. 

It follows from the definition \eqref{eq:kl} of relative entropy that $H(\mu \otimes \nu | \mu \otimes \nu) = 0$. In particular, setting $\Pi = \mu \otimes \nu$ in the right-hand side of each of \eqref{eq:max1}, \eqref{eq:max2} and \eqref{eq:max3}
we obtain the following result.
\begin{cor} \label{cor:ineq}
Let $\mu$ and $\nu$ be probability measures on the real line with compact support.
In each of the following equations, under $\mathbf{E}_{\mu \otimes \nu}$ let $X$ and $Y$ be independent random variables with marginal laws $\mu$ and $\nu$. 

Then for $z > E_\mu^+ + E_\nu^+$ we have
\begin{align} \label{eq:maxa1}
\int_{-\infty}^\infty \log(z-x)(\mu \boxplus \nu)(\mathrm{d}x) \geq \mathbf{E}_{\mu \otimes \nu}[\log(z - (X+Y))].
\end{align}
If we additionally suppose that $\nu$ is supported on $[0,\infty)$, then whenever $z > \max \{ E_\mu^+ E_\nu^+, E_\mu^- E_\nu^-\} $ we have 
\begin{align} \label{eq:maxa2}
\int_{-\infty}^\infty \log(z-x)(\mu \boxtimes \nu)(\mathrm{d}x) \geq \mathbf{E}_{\mu \otimes \nu}[\log(z - XY)].
\end{align}
Finally, for $z > E_\mu^+$ we have 
\begin{align} \label{eq:maxa3}
\int_{-\infty}^\infty \log(z-x)[\mu]_\tau (\mathrm{d}x) \geq \mathbf{E}_{\mu}[\log(z-X)].
\end{align}
\end{cor}

Since the function $x \mapsto \log(z-x)$ is concave, the equations \eqref{eq:maxa1}, \eqref{eq:maxa2} and \eqref{eq:maxa3} may be regarded as somehow encoding the fact that free operations produce measures with thinner tails than their classical counterparts, subject to having the same variance. 

A recent preprint \cite{heinavaara} by Otte Heinävaara shows that \eqref{eq:maxa1} continues to hold when the function $x \mapsto \log(z-x)$ is replaced by any function $x \mapsto f(x)$ whose fourth derivative is nonpositive, and that this condition is essentially sharp. 
The proof relies on free cumulant manipulations together with tools from earlier work of Heinävaara \cite{heinavaara2}.

\subsection{The multivariate case}
In Theorem \ref{thm:main} we discussed how one can formulate the free additive or multiplicative convolution of two different probability measures in terms of a variational problem. In fact, in the sequel we will prove the following multivariate generalisation of the first two cases of Theorem \ref{thm:main}. 

\begin{thm} \label{thm:mainmulti}
Let $\mu_1,\ldots,\mu_d$ be probability measures on the real line with compact support. Then for $z > \sum_{j=1}^d E_{\mu_j}^+$ we have
\begin{align} \label{eq:max1multi}
\int_{-\infty}^\infty \log(z-x) (\mu_1 \boxplus \cdots \boxplus \mu_d ) (\mathrm{d}x) = \sup_\Pi \left\{ \mathbf{E}_\Pi \left[ \log \left( z - \sum_{j=1}^d X_j \right) \right] - H(\Pi|\otimes_{j=1}^d \mu_j) \right\}.
\end{align}

Suppose in addition that $\mu_2,\ldots,\mu_d$ are supported on $[0,\infty)$. 
Then, for $z > \max \prod_{j=1}^d E_{\mu_j}^{\pm} $
we have
\begin{align} \label{eq:max2multi}
\int_{-\infty}^\infty \log(z-x) (\mu_1 \boxtimes \cdots \boxtimes \mu_d ) (\mathrm{d}x)
= \sup_\Pi \left\{ \mathbf{E}_\Pi \left[ \log \left( z - \prod_{j=1}^d X_j \right) \right] - H(\Pi \mid \otimes_{j=1}^d \mu_j) \right\}.
\end{align}

In each of \eqref{eq:max1multi} and \eqref{eq:max2multi}, the supremum is taken over all couplings $\Pi$ of $\mu_1,\ldots,\mu_d$.
\end{thm}

\subsection{Cauchy transforms and optimal couplings}

 Let us introduce $\boxdot := \boxplus/\boxtimes$ as a shorthand for either additive or multiplicative free convolution, and let us introduce the corresponding notation $\odot = +/\times$ for the addition or multiplication of real numbers or real-valued random variables.
As in Section \ref{sec:explicit}, in both \eqref{eq:max1multi} and \eqref{eq:max2multi} there exists a unique coupling $\Pi_{*,z}$ attaining the supremum, and this is the unique coupling whose Radon--Nikodym derivative with respect to the product measure takes the form
\begin{align} \label{eq:RN}
\frac{\mathrm{d} \Pi_{*,z} }{ \mathrm{d}(\mu_1 \otimes \cdots \otimes \mu_d) } 
= \omega(z) \frac{ z - X_1 \odot \cdots \odot X_d }{ \prod_{i=1}^d (\omega_i(z) - X_i) } 
=: W^{\odot}_z(X_1,\ldots,X_d),
\end{align}
where $\odot$ denotes addition or multiplication, depending on the choice of convolution. 
Equation \eqref{eq:RN} is proved in Theorem \ref{thm:explicitnew}.

The following result develops Theorem \ref{thm:mainmulti} further, stating that the multivariate subordination equations for sufficiently large real $z$ are encoded by the optimal coupling $\Pi_{*,z}$. 
More precisely, it shows that the Cauchy transform at $z$ of $\mu_1 \boxdot \cdots \boxdot \mu_d$ coincides with that of $X_1 \odot \cdots \odot X_d$ under $\Pi_{*,z}$.
\begin{proposition} \label{prop:cauchynew}
Let $z > \sum_{j=1}^d E_{\mu_j}^+$. Then the Cauchy transform of $\mu_1 \boxplus \cdots \boxplus \mu_d$ at $z$ is given by 
\begin{align} \label{eq:max1bnew}
\int_{-\infty}^\infty \frac{1}{z-x} (\mu_1 \boxplus \cdots \boxplus \mu_d)(\mathrm{d}x) = \mathbf{E}_{\Pi_{*,z}} \left[ \frac{1}{z - (X_1 + \cdots + X_d) } \right],
\end{align}
where $\Pi_{*,z}$ is defined via \eqref{eq:RN} with $\odot = +$. 

Let $z > \max \prod_{j=1}^d E_{\mu_j}^{\pm}$. Then the Cauchy transform of $\mu_1 \boxtimes \cdots \boxtimes \mu_d$ at $z$ is given by 
\begin{align} \label{eq:max2bnew}
 \int_{-\infty}^\infty \frac{1}{z-x} (\mu_1 \boxtimes \cdots \boxtimes \mu_d)(\mathrm{d}x) = \mathbf{E}_{\Pi_{*,z}} \left[ \frac{1}{z - X_1 \cdots X_d } \right],
\end{align}
where $\Pi_{*,z}$ is defined via \eqref{eq:RN} with $\odot = \times$.

Finally, let $z > E^+_\mu$. If $\Pi_{*,z}$ is the unique coupling achieving the supremum in \eqref{eq:max3} (see \eqref{eq:RN4} for its precise form), then the Cauchy transform of $[\mu]_\tau$ at $z$ is given by 
\begin{align} \label{eq:max3b}
 \int_{-\infty}^\infty \frac{1}{z-x} [ \mu]_\tau (\mathrm{d}x) = \frac{1}{\tau} \mathbf{E}_{\Pi_{*,z}} \left[ \frac{1}{z-X} \mathrm{1}_{\{Y > 1-\tau\}} \right].
\end{align}
\end{proposition}

Proposition \ref{prop:cauchynew} is proved in Section \ref{sec:variational}.

\subsection{Subordination equations through complex couplings}
The analogue of Proposition \ref{prop:cauchynew} for complex $z$ provides a framework for converting the results of classical probability into those of free probability. 
This is the topic of forthcoming work \cite{AJ} by the authors, but we outline these ideas here. 
Let $z$ be a sufficiently large real number. Since $\Pi_{*,z}$ defined via \eqref{eq:RN} is a coupling of $(X_1,\ldots,X_d)$, we have that for any Borel subset $E \subset \mathbb{R}$ and for each $1 \leq i \leq d$,
\begin{align*}
\mathbf{E}_{\mu_1 \otimes \cdots \otimes \mu_d} \left[ W_z^{\odot}(X_1,\ldots,X_d)\, \mathrm{1}_{\{ X_i \in E \}} \right] = \mu_i(E).
\end{align*}
This holds for all $E$ if and only if the consistency condition
\begin{align} \label{eq:consistency}
\mathbf{E}_{\mu_1 \otimes \cdots \otimes \mu_d} \left[ W_z^{\odot}(X_1,\ldots,X_d) \mid X_i = x_i \right] = 1 \qquad \text{for $\mu_i$-a.e. $x_i \in \mathbb{R}$}
\end{align}
is satisfied.
The relation \eqref{eq:consistency} reduces to a collection of algebraic equations involving the Cauchy transforms of $\mu_1,\ldots,\mu_d$. We will see in our calculations in Section \ref{sec:variational} that these equations are precisely the subordination equations. 
While in our framework we have studied couplings of random variables and have so far assumed that $z$ is a sufficiently large real number, it is possible to continue to make sense of the change of measure \eqref{eq:RN} and Proposition \ref{prop:cauchynew} when $z$ is complex.

Recalling that $\boxdot = \boxplus/\boxtimes$ denotes either additive or multiplicative free convolution, and that $\odot = +/\times$ denotes the corresponding operation on real numbers or real-valued random variables, in \cite{AJ} the following extension to $z \in \mathbb{C}^+ := \{ z \in \mathbb{C} : \mathrm{Im}(z) > 0 \}$ of the first two cases of Proposition \ref{prop:cauchynew} is proved.

\begin{proposition}[\cite{AJ}] \label{prop:cauchynew2}
Let $z \in \mathbb{C}^+$ and let $\mu_1,\ldots,\mu_d$ be probability measures.
\begin{enumerate}
\item For each $z \in \mathbb{C}^+$, there exist unique elements $\omega(z),\omega_1(z),\ldots,\omega_d(z)$ in $\mathbb{C}^+$ such that the complex-valued random variable $W_z^\odot(X_1,\ldots,X_d)$ of the form given in \eqref{eq:RN} satisfies the consistency conditions \eqref{eq:consistency} for each $1 \leq i \leq d$.
\item For each $z \in \mathbb{C}^+$, we have the relation
\begin{align} \label{eq:cauchynew2}
\int_{-\infty}^\infty \frac{1}{z-x} (\mu_1 \boxdot \cdots \boxdot \mu_d)(\mathrm{d}x) = \mathbf{E}_{\mu_1 \otimes \cdots \otimes \mu_d} \left[ \frac{1}{z - X_1 \odot \cdots \odot X_d} W_z^\odot(X_1,\ldots,X_d) \right],
\end{align}
where $W_z^\odot(X_1,\ldots,X_d)$ is as in (1).
\end{enumerate}
\end{proposition}
Proposition \ref{prop:cauchynew2} follows from the analytic continuation of Proposition \ref{prop:cauchynew} to the upper half-plane.

The value of Proposition \ref{prop:cauchynew2} becomes particularly apparent when $\mu_1 = \cdots = \mu_d$, where it may be used to convert classical results about the stability, divisibility, and limiting properties of sums of random variables to their free analogues. To provide an example here, we now sketch how Proposition \ref{prop:cauchynew2} may be used to derive the free central limit theorem from the classical central limit theorem (in the setting $\boxdot = \boxplus$ and $\odot = +$).
Let $\mu$ be a probability measure with mean zero, unit variance, and bounded support. 
Let $\mu_{(d)}$ be the law of $X/\sqrt{d}$, where $X \sim \mu$. 
If $\mu_1 = \cdots = \mu_d = \mu_{(d)}$, then $G_d := X_1+\cdots+ X_d$ is approximately a standard Gaussian under $\mu_1 \otimes \cdots \otimes \mu_d$. 
In particular, noting that $\omega_1(z) = \cdots = \omega_d(z) =: \omega_0(z)$ are identical when $\mu_1 = \cdots = \mu_d$, by setting $\odot = +$ in \eqref{eq:RN}, when $d$ is large we have
\begin{align} \label{eq:XX}
W_z^+(X_1,\ldots,X_d) = \omega(z) \frac{ z - G_d}{ \prod_{i=1}^d (\omega_0(z) - X_i) } \approx c(z) (z - G_d) e^{f(z)G_d} =: W_z(G_d),
\end{align}
where $f(z) = 1/\omega_0(z)$ and $c(z) = \omega(z)/\omega_0(z)^d$. 
In particular, using the heuristic \eqref{eq:XX} in \eqref{eq:cauchynew2}, we obtain the relation
\begin{align} \label{eq:derv}
\lim_{d \to \infty} \int_{-\infty}^\infty \frac{1}{z-x} (\mu_{(d)}^{\boxplus d})(\mathrm{d}x) = \mathbf{E} \left[ \frac{W_z(G)}{z-G} \right],
\end{align}
where $G$ is a standard Gaussian random variable, and $W_z(G)$ is the unique random variable of the form $W_z(G) = c(z) (z-G)e^{f(z)G}$ with $f(z)$ in the lower half-plane, satisfying $\mathbf{E}[W_z(G)] = 1$ and $\mathbf{E}[GW_z(G)] = 0$. 
These latter equations are the asymptotic residues of the consistency equations. 

One can solve for $c(z)$ and $f(z)$ (the latter necessarily in the lower half-plane) explicitly, and conclude that 
\begin{align} \label{eq:solver}
\mathbf{E} \left[ \frac{W_z(G)}{z-G} \right] = \frac{z-\sqrt{z^2-4}}{2} = \int_{-\infty}^\infty \frac{1}{z-x} \frac{1}{2\pi} \sqrt{4 - x^2} \mathrm{1}_{[-2,2]}(x)\mathrm{d}x;
\end{align}
see Appendix \ref{sec:solver} for the proof of the first equality in \eqref{eq:solver}.
Thus the coupling description of the subordination equations in Proposition \ref{prop:cauchynew2} can be used to derive the free central limit theorem from its classical counterpart.

The relations \eqref{eq:derv} and \eqref{eq:solver} are special cases of a new formula established in \cite{AJ} for the Bercovici--Pata bijection \cite{BPB} relating classically infinitely divisible probability measures $\nu$ to freely infinitely divisible probability measures $\nu'$. 

\subsection{Absolute values and $z$ contained in the support}

In the additive case, the reader may wonder whether we can extend the range of $z$ in the equalities \eqref{eq:max1}-\eqref{eq:max3} to encompass all real $z$, simply replacing $\log(z-x)$ with $\log|z-x|$. In general, this is not the case. That is, in general
\begin{align} \label{eq:max14}
\int_{-\infty}^\infty \log|z-x| ( \mu \boxplus \nu )(\mathrm{d}x) \neq \sup_{\Pi} \left\{ \mathbf{E}_\Pi \left[ \log |z - (X+Y)| \right] - H(\Pi|\mu \otimes \nu) \right\}, \qquad z \in \mathbb{R}.
\end{align}
We prove that equality does not hold in the setting of \eqref{eq:max14} by way of a counterexample; see Appendix \ref{sec:counterexample}.

\subsection{Interpolation between free and classical convolution}
Let $\mu$ and $\nu$ be compactly supported probability measures on $\mathbb{R}$. Consider introducing an inverse temperature parameter $\beta$ to the entropic optimal transport problem, so that for $z > E_\mu^+ + E_\nu^+$ we consider 
\begin{align*}
F_\beta(z) := \sup_{\Pi} \left\{ \mathbf{E}_\Pi \left[ \log (z - (X+Y)) \right] - \beta H(\Pi \mid \mu \otimes \nu) \right\}.
\end{align*}
The additive case of Theorem \ref{thm:main} states that 
\[
F_1(z) = \int_{-\infty}^\infty \log (z-x)\, (\mu \boxplus \nu)(\mathrm{d}x).
\]
Conversely, one readily verifies that 
\[
\lim_{\beta \to \infty} F_\beta(z) = \int_{-\infty}^\infty \log (z-x)\, (\mu \ast \nu)(\mathrm{d}x),
\]
since the independent coupling of $X$ and $Y$ is the unique coupling minimising the relative entropy with respect to $\mu \otimes \nu$. 

We pose the following problem:

\begin{prob} \label{prob:interpolation}
Let $\mu$ and $\nu$ be compactly supported probability measures. For which $\beta \geq 0$ does there exist a probability measure $\mu \boxplus_\beta \nu$ such that for all $z > E_\mu^+ + E_\nu^+$, we have
\begin{align*}
F_\beta(z) = \int_{-\infty}^\infty \log (z - x)\, (\mu \boxplus_\beta \nu)(\mathrm{d}x)?
\end{align*}
\end{prob}

As observed preceding the statement of Problem \ref{prob:interpolation}, $\mu \boxplus_1 \nu$ coincides with free convolution, and $\mu \boxplus_\infty \nu$ coincides with classical convolution. 
Thus, if it is the case (for some choices $\mu$ and $\nu$) that $\mu \boxplus_\beta \nu$ exists for all $\beta \in [1,\infty)$, then this family of probability measures may be regarded as interpolating between free and classical convolution. It would then be interesting to compare this interpolation with other forms appearing in the literature (e.g.\ \cite{BLG}).

The problem also has interesting content for $\beta \in [0,1)$. Let us highlight that for any $\mu$ and $\nu$ with compact support, one can prove that $\mu \boxplus_0 \nu$ corresponds to the countermonotone coupling of $\mu$ and $\nu$. In fact, the authors have verified (in fairly lengthy computations which we elect not to supply here) that, in the special case where $\mu = \nu = \frac{1}{2}\delta_0 + \frac{1}{2}\delta_1 =: \pi$, the measure $\pi \boxplus_\beta \pi$ exists for $\beta \geq 0$ if and only if $\beta \in \{0\} \cup [1/2,\infty)$. 
In particular, we find that $
\pi \boxplus_{1/2} \pi = \frac{1}{2} \delta_{1 - 1/\sqrt{2}} + \frac{1}{2} \delta_{1 + 1/\sqrt{2}}$.

\subsection{The Schrödinger dual formulation}

In entropic optimal transport we have the duality relation
\begin{align*}
&\sup_{\Pi} \Bigl\{ \mathbf{E}_{\Pi}\!\left[c(X,Y)\right] - H\!\left(\Pi \,\middle|\, \mu\otimes \nu\right) \Bigr\}\\
&=
\inf_{f,g} \left\{ \mathbf{E}_{\mu}[f(X)] + \mathbf{E}_{\nu}[g(Y)] \;+\; \log \mathbf{E}_{\mu\otimes \nu}\!\left[\exp\!\bigl(c(X,Y)-f(X)-g(Y)\bigr)\right] \right\},
\end{align*}
where here the infimum is simply taken over all functions $f$ and $g$ for which the relevant expectations exist; see Section 3 of Nutz \cite{nutz}. Applying this relation to the additive setting \eqref{eq:max1} of our main result, we obtain as a corollary the following alternative formulation of \eqref{eq:max1}.

\begin{cor} \label{cor:dual}
For all $z > E_\mu^+ + E_\nu^+$ we have
\begin{align*}
&\int_{-\infty}^\infty \log(z-x) ( \mu \boxplus \nu) (\mathrm{d}x) \\
&= \inf_{f,g} \left\{ \mathbf{E}_{\mu}[f(X)] + \mathbf{E}_{\nu}[g(Y)] \;+\; \log \mathbf{E}_{\mu\otimes \nu}\!\left[(z-(X+Y))e^{-f(X)-g(Y)}\right] \right\},
\end{align*}
where the infimum is taken over all $f,g:\mathbb{R} \to \mathbb{R}$ for which the expectations exist.
\end{cor}

There are, of course, analogues for Corollary \ref{cor:dual} in the setting of \eqref{eq:max2} and \eqref{eq:max3}, and in multivariate settings.

The infimum in Corollary \ref{cor:dual} is attained uniquely when
\begin{align*}
f(X) + g(Y) = \log(\omega_\mu(z) -X) + \log(\omega_\nu(z) - Y) -\log \omega(z),
\end{align*}
where $\omega_\mu(z),\omega_\nu(z),\omega(z)$ are as in \eqref{eq:add10} and \eqref{eq:add20}; c.f.\ \eqref{eq:apollo20}.

\subsection{Atoms at the top of support} 

Unlike classical convolution, free convolution of measures conspires to reduce the appearance of atoms where possible. 
Short of giving a full account of this phenomenology here (the reader might consult, e.g., the works \cite{belinschi, belinschi2} of Belinschi), we will be content to sketch how the following observation drops out rapidly from our main result:

\begin{proposition} \label{prop:atom}
Suppose that $\mu$ and $\nu$ are supported on $[E_\mu^-,E_\mu^+]$ and $[E_\nu^-,E_\nu^+]$ respectively, and that $\mu$ has an atom of size $a$ at $E_\mu^+$ and $\nu$ has an atom of size $b$ at $E_\nu^+$. Then
$\mu \boxplus \nu$ has an atom of size $\max\{ a+b-1, 0 \}$ at $E_\mu^+ + E_\nu^+$.
\end{proposition}

\begin{proof}[Sketch proof]
Let $\Pi$ be a coupling of random variables $X$ and $Y$ with respective laws $\mu$ and $\nu$. 
Consider the asymptotic behaviour as $\varepsilon \to 0$ of the relation in \eqref{eq:max1} with $z = z_\varepsilon = E_\mu^+ + E_\nu^+ + \varepsilon$. 
On the one hand, the left-hand side of \eqref{eq:max1} takes the form 
\begin{align} \label{eq:ep1}
\int_{-\infty}^\infty \log (z_\varepsilon - x) (\mu \boxplus \nu)(\mathrm{d}x) = (1+o(1))\log(\varepsilon)p_{\boxplus} \qquad \text{as $\varepsilon \to 0$},
\end{align}
where $p_{\boxplus}$ is the size of the atom of $\mu \boxplus \nu$ at $E_\mu^+ + E_\nu^+$. 
On the other hand, if $\Pi$ is a coupling of $X \sim \mu$ and $Y \sim \nu$, write $p_\Pi := \mathbf{P}_\Pi( X = E_\mu^+, Y= E_\nu^+)$. 
Then with $z = z_\varepsilon$, the right-hand side of \eqref{eq:max1} takes the form
\begin{align} \label{eq:ep2}
 (1 + o(1))  \sup_\Pi \left\{ \log(\varepsilon) p_\Pi\right\},\qquad \text{as $\varepsilon \to 0$}.
\end{align} 
In particular, as $\varepsilon \downarrow 0$, the optimal coupling minimises $p_\Pi$. The minimum such probability is $\max \{a + b - 1,0\}$. 
It follows from relating \eqref{eq:ep1} and \eqref{eq:ep2} that this is also the size of the atom at $E_\mu^+ + E_\nu^+$ in free convolution $\mu \boxplus \nu$, completing the proof.
\end{proof}

There are analogues to Proposition \ref{prop:atom} for multiplicative free convolution and free compression. 

\subsection{Asymptotic versions of eigenvalue and determinant inequalities}
As discussed in Section \ref{sec:overview}, our main result Theorem \ref{thm:main} arises out of an asymptotic treatment of the quadrature formula \eqref{eq:GM0}, which provides implicit information about the eigenvalues of a sum of unitarily invariant matrices $A$ and $B$ in terms of the eigenvalues of $A$ and $B$. In parallel to this idea, we touch on here how it should be possible to obtain asymptotic variants of 
various (deterministic) eigenvalue and determinant inequalities.

Take for instance, Fiedler's inequality \cite{fiedler}, which states that for real diagonal matrices $A$ and $B$ with diagonal entries $a_1,\ldots,a_N$ and $b_1,\ldots,b_N$ we have
\begin{align} \label{eq:fiedler}
\min_{ \sigma \in \mathcal{S}_N } \prod_{i=1}^N (z - (a_i + b_{\sigma(i)})) \leq \det(z - (A + UBU^{-1})) \leq \max_{ \sigma \in \mathcal{S}_N } \prod_{i=1}^N (z - (a_i + b_{\sigma(i)})).
\end{align}

Consider now in the setting of \eqref{eq:fiedler} letting $(A_N)_{N \geq 1}$ and $(B_N)_{N \geq 1}$ be sequences of matrices satisfying \eqref{eq:conv00}, and let $(U_N)_{N \geq 1}$ be a sequence of Haar unitary random matrices. By taking $\frac{1}{N} \log$-asymptotics of \eqref{eq:fiedler} in such a regime, one can obtain the macroscopic Fiedler inequality 
\begin{align} \label{eq:hydrofiedler}
\inf_\Pi \mathbf{E}_{\Pi}[ \log(z-(X+Y))]  &\leq \int_{-\infty}^\infty \log(z - x) \mu \boxplus \nu(\mathrm{d}x)\leq \sup_\Pi \mathbf{E}_{\Pi}[ \log(z-(X+Y))], 
\end{align}
where the infimum and supremum are taken over all couplings $\Pi$ of $\mu$ and $\nu$. The upper bound in \eqref{eq:hydrofiedler} is a weaker version of \eqref{eq:max1}. 

Similarly, one can also formulate asymptotic versions of the Horn inequalities \cite{horn, KT}, which describe the possible eigenvalues of a Hermitian matrix equation $A+B=C$. For example, one such asymptotic inequality reads $T_{\mu \boxplus \nu}(s+t-1) \leq T_\mu(s) + T_\nu(t)$ whenever $s+t \geq 1$. Here $T_\mu:[0,1] \to \mathbb{R}$ is the quantile function of a probability measure $\mu$. See also \cite{BL}.


\subsection{Large-$z$ asymptotics}
By studying the large-$z$ asymptotics of the formulas \eqref{eq:max1}, \eqref{eq:max2} and \eqref{eq:max3}, one can obtain information about the moments of free operations on measures.

Take for example \eqref{eq:max1}. Using the expansion $\log(z-x) = \log z - \sum_{k \geq 1} \frac{1}{k} (x/z)^k$ in \eqref{eq:max1}, as well as the fact that the expectation $\mathbf{E}_{\Pi}[ \log z - \frac{1}{z}(X+Y) ]$ of the first two terms in the expansion does not depend on the coupling $\Pi$ of $\mu$ and $\nu$, we obtain first of all that $\int_{-\infty}^\infty x (\mu \boxplus \nu)(\mathrm{d}x) = \int_{-\infty}^\infty x \mu(\mathrm{d}x) + \int_{-\infty}^\infty x \nu(\mathrm{d}x)$, as well as the equation
\begin{align} \label{eq:max1b0}
\sum_{k \geq 2} \frac{1}{kz^k} \int_{-\infty}^\infty x^k ( \mu \boxplus \nu )(\mathrm{d}x) =\inf_{\Pi} \left\{ \sum_{k \geq 2} \frac{1}{kz^k}\mathbf{E}_\Pi[(X+Y)^k] + H(\Pi|\mu \otimes \nu) \right\}
\end{align}
relating the higher moments. 

Note that when $z$ is large, the entropy penalty inside the infimum on the right-hand side of \eqref{eq:max1b0} is large compared to the sum. Thus there is a large incentive for the optimal coupling to be close to product measure. 
With this in mind, it is a brief exercise to show that the unique coupling $\Pi_{*,z}$ achieving the minimum in \eqref{eq:max1b0} takes the form $\mathrm{d}\Pi_{*,z}/\mathrm{d}(\mu \otimes \nu) = 1 + \frac{1}{z^2}Q_z$ as $z \to \infty$, where $Q_z$ is a random variable satisfying $\mathbf{E}_{\mu \otimes \nu}[Q_z] = 0$ and converging to a limit random variable $Q_\infty$ as $z \to \infty$. As a result, one can deduce that the first three moments of additive free convolution are the same as those of classical convolution, but the fourth moment is smaller. We leave the details to the interested reader.

\section{Quadrature formulas, couplings and copulas} \label{sec:3}

In this section we set up the proofs of our main results, Theorem \ref{thm:main} and Theorem \ref{thm:mainmulti}. Note that the first two statements of Theorem \ref{thm:main} are the case $d=2$ of Theorem \ref{thm:mainmulti}. 

\subsection{Couplings, copulas, and relative entropy}
In order to prepare our work for the remainder of the article, it is useful to discuss in full the connection between couplings of probability measures and their associated copulas, and how this relationship pertains to relative entropy. 

Let $\mu_1,\ldots,\mu_d$ be probability measures on $\mathbb{R}$. A \textbf{coupling} of $\mu_1,\ldots,\mu_d$ is a probability measure on $\mathbb{R}^d$ with the property that its pushforward under the $j^{\text{th}}$ coordinate projection is $\mu_j$. 
A \textbf{copula} on $[0,1]^d$ is simply a coupling of $\mu_1 = \cdots = \mu_d = \nu_{\mathrm{uni}}$, where $\nu_{\mathrm{uni}}$ is Lebesgue measure on $[0,1]$. A copula may or may not have a density with respect to Lebesgue measure on $[0,1]^d$. If a copula does have a density $\gamma:[0,1]^d \to \mathbb{R}$ with respect to Lebesgue measure, this density satisfies
\begin{align} \label{eq:jetski}
\int_{[0,1]^{d-1}}\gamma(s_1,\ldots,s_d) \mathrm{d}s_1 \ldots \mathrm{d}s_{j-1} \mathrm{d}s_{j+1}\ldots \mathrm{d}s_d = 1
\end{align}
for all $s_j \in [0,1]$. A \textbf{permuton} is simply a copula in the setting $d=2$, though for the sake of consistency even in the case $d=2$ we will refer to permutons as copulas.

Given a probability measure $\mu$ on the real line, we write $T_\mu:[0,1] \to \mathbb{R} \cup \{ \pm \infty\}$ for its \textbf{quantile function}. This is the left-continuous function satisfying
\begin{align} \label{eq:quantile}
\int_{\mathbb{R}} f(x) \mu(\mathrm{d}x) = \int_{(0,1)} f(T_\mu(t))\mathrm{d}t.
\end{align}
The values $T_\mu(0)$ and $T_\mu(1)$ are defined by taking right and left limits respectively.
If $U$ is uniformly distributed on $[0,1]$, then $T_\mu(U)$ is distributed according to $\mu$. If $\mu$ has a gap in its support of length $g$, $T_\mu:[0,1] \to \mathbb{R}$ has a jump of size $g$. If $\mu$ has a Dirac mass of size $m$, then $T_\mu:[0,1] \to \mathbb{R}$ is constant on a subinterval of $[0,1]$ of length $m$. 

Recall that if $\mu$ is a measure on $E$, and $T:E \to F$ is a measurable map, we may define the pushforward measure $T_\# \mu$ on $F$ by setting $T_\# \mu (F' ) := \mu(T^{-1}(F'))$ for measurable subsets $F'$ of $F$. 

Given probability measures $\mu_1,\ldots,\mu_d$ on $\mathbb{R}$, every copula $\gamma$ on $[0,1]^d$ gives rise to a coupling $\Pi$ of the probability measures $\mu_1,\ldots,\mu_d$ by letting $\Pi$ denote the law of the vector $(T_{\mu_1}(U_1),\ldots,T_{\mu_d}(U_d))$ where $(U_1,\ldots,U_d)\sim \gamma$. In other words,
\begin{align*}
\Pi = T_\# \gamma,
\end{align*}
where $T:[0,1]^d \to \mathbb{R}^d$ is given by 
\begin{align} \label{eq:ton1}
T(s_1,\ldots,s_d) := (T_{\mu_1}(s_1),\ldots,T_{\mu_d}(s_d)).
\end{align}
In fact, every coupling $\Pi$ of $\mu_1,\ldots,\mu_d$ arises this way \cite{sklar}.

When each of $\mu_1,\ldots,\mu_d$ has a density with respect to Lebesgue measure, the map $T$ is injective, and the correspondence between $\Pi$ and $\gamma$ is bijective in that they uniquely determine one another. Otherwise it is possible that different copulas may give rise to the same coupling of $\mu_1,\ldots,\mu_d$. 

Take for example, in the setting $d=2$, two probability measures $\mu$ and $\nu$ both of which contain Dirac masses. Then $T_\mu$ (resp.\ $T_\nu$) is constant on a non-empty subinterval $(a,b)$ (resp.\ $(c,d)$) of $[0,1]$. Then any pair of copulas $\gamma$ and $\gamma'$ that agree on $[0,1]^2 - (a,b) \times (c,d)$ and give the same mass to $(a,b) \times (c,d)$ give rise to the same coupling of $\mu$ and $\nu$, despite possibly being different measures on the rectangle $(a,b) \times (c,d)$.

We say that a copula $\gamma$ is \textbf{flat} with respect to quantile functions $(T_{\mu_1},\ldots,T_{\mu_d})$ if, for all intervals $(a_j,b_j)$ such that $T_{\mu_j}$ is constant on $(a_j,b_j)$, $\gamma$ coincides with a multiple of the Lebesgue measure on the box $\prod_{j=1}^d (a_j,b_j)$. There is a bijective correspondence between couplings $\Pi$ of $\mu_1,\ldots,\mu_d$ and copulas that are flat with respect to $(T_{\mu_1},\ldots,T_{\mu_d})$.

Recall from the introduction that for probability measures $\Pi$ and $\Pi'$ on $E$ with $\Pi'$ absolutely continuous against $\Pi$, we write $H(\Pi'|\Pi) := \int_E  \log \frac{ \mathrm{d}\Pi'}{\mathrm{d}\Pi} ~\mathrm{d}\Pi'$ for the relative entropy of $\Pi'$ against $\Pi$. 

If a copula $\gamma$ on $[0,1]^d$ has a density with respect to Lebesgue measure, we will sometimes abuse notation and write $\gamma:[0,1]^d \to \mathbb{R}$ for the density itself. Given a probability measure $\gamma$ on $[0,1]^d$, write
\begin{align} \label{eq:Hdef}
H(\gamma) := 
\begin{cases}
\int_{[0,1]^d} \gamma(s) \log \gamma(s) \mathrm{d}s \qquad &\text{$\gamma$ is a copula \& has density w.r.t.\ Leb.}\\
+ \infty \qquad &\text{otherwise}. 
\end{cases}
\end{align}
We emphasise that if $\gamma$ is a copula on $[0,1]^d$ but not absolutely continuous with respect to Lebesgue measure, then $H(\gamma) =+\infty$. Note that if $\gamma$ is a copula, $H(\gamma) = H(\gamma|\nu_{\mathrm{uni}}^{\otimes d})$, where $\nu_{\mathrm{uni}}^{\otimes d}$ is Lebesgue measure on $[0,1]^d$.

\begin{lemma} \label{lem:flat}
Let $\gamma$ be a copula on $[0,1]^d$, and let $\Pi := T_\# \gamma$ be the associated coupling of $\mu_1,\ldots,\mu_d$, where $T$ is as in \eqref{eq:ton1}. Then 
\begin{align*}
H(\gamma) \geq H(\Pi|\otimes_{j=1}^d \mu_j),
\end{align*}
with equality if and only if $\gamma$ is flat.
\end{lemma}
\begin{proof}
Relative entropy decreases under joint pushforwards, namely, whenever $T:E\to F$ is a measurable map, we have
\begin{align} \label{eq:relpush}
H(T_\# \pi'|T_\# \pi) \leq H(\pi'|\pi),
\end{align}
with equality if and only if $\mathrm{d}\pi'/\mathrm{d}\pi$ is constant on the fibers $T^{-1}(y)$ for $y \in F$. 
See \cite[Section 2]{lehec} or \cite[Section 10]{varadhan} for details. 

Note that $\Pi = T_\# \gamma$, and $\otimes_{j=1}^d \mu_j = T_\# \nu_{\mathrm{uni}}^{\otimes d}$. Then using \eqref{eq:relpush} to obtain the inequality below, we have 
\begin{align*}
H(\Pi|\otimes_{j=1}^d \mu_j)=H(T_\# \gamma | T_\# \nu_{\mathrm{uni}}^{\otimes d}) \leq H(\gamma | \nu_{\mathrm{uni}}^{\otimes d} ) = H(\gamma)
\end{align*}
with equality occurring only when $\gamma = \mathrm{d}\gamma/\mathrm{d}\nu_{\mathrm{uni}}^{\otimes d}$ is constant on the fibers $T^{-1}(x_1,\ldots,x_d)$ of $T$, i.e.\ when $\gamma$ is flat.
\end{proof}

We are now equipped to express Theorem \ref{thm:main} and its multivariate version Theorem \ref{thm:mainmulti} in the language of copulas and quantile functions.

\begin{thm}[Restatement of Theorem \ref{thm:main} and Theorem \ref{thm:mainmulti}]\label{thm:main2}
Let $\mu_1,\ldots,\mu_d$ be probability measures on the real line with compact support. 

Then, for $z > \sum_{j=1}^d E_{\mu_j}^+$ we have 
\begin{align} \label{eq:max1s}
 \int_{-\infty}^\infty \log(z-x) ( \mu_1 \boxplus \cdots \boxplus \mu_d )(\mathrm{d}x)  =\sup_{\gamma} \left\{ \int_{[0,1]^d}  \log \left( z - \sum_{j=1}^d T_{\mu_j}(s_j) \right) \gamma(s)\mathrm{d}s - H(\gamma) \right\}.
\end{align}
Suppose we additionally assume that $\mu_2,\ldots,\mu_d$ are supported on $[0,\infty)$. Then for $z > \max \prod_{j=1}^d E_{\mu_j}^\pm$ we have 
\begin{align} \label{eq:max2s}
 \int_{-\infty}^\infty \log(z-x) ( \mu_1 \boxtimes \cdots \boxtimes \mu_d )(\mathrm{d}x)  =\sup_{\gamma} \left\{ \int_{[0,1]^d}  \log \left( z - \prod_{j=1}^d T_{\mu_j}(s_j) \right) \gamma(s)\mathrm{d}s - H(\gamma) \right\}.
\end{align}
Finally, for $z > E^+_{\mu_1}$ we have 
\begin{align} \label{eq:max3s}
\tau \int_{-\infty}^\infty  \log(z-x)  [\mu_1]_\tau (\mathrm{d}x) =\sup_{\gamma} \left\{ \int_{[0,1]^2}  \log(z-T_{\mu_1}(s_1))\mathrm{1}_{\{s_2 > 1-\tau \}} \gamma(s)\mathrm{d}s - H(\gamma) \right\}.
\end{align}
In \eqref{eq:max1s} and \eqref{eq:max2s}, the suprema are taken over all copulas on $[0,1]^d$. In \eqref{eq:max3s}, the supremum is taken over all copulas on $[0,1]^2$.
\end{thm}
\begin{proof}[Proof of equivalence between this statement and Theorems \ref{thm:main} and Theorem \ref{thm:mainmulti}]
We prove the equivalence between \eqref{eq:max1s} and \eqref{eq:max1multi}; the proofs of the other equivalences are similar. First note that if $\gamma$ is a copula associated with a coupling $\Pi$ of $\mu_1,\ldots,\mu_d$, i.e.\ $\Pi = T_\# \gamma$, then by definition
\begin{align*}
\mathbf{E}_\Pi\left[ \log\left(z-\sum_{j=1}^d X_j\right) \right] = \int_{[0,1]^d} \log \left(z-\sum_{j=1}^d T_{\mu_j}(s_j) \right) \gamma(\mathrm{d}s) .
\end{align*}
Now note that
\begin{align*}
\int_{[0,1]^d}  \gamma(s) \log \gamma(s) \mathrm{d}s= H(\gamma) \geq H(\Pi|\otimes_{j=1}^d \mu_j),
\end{align*}
with equality whenever $\gamma$ is flat. Thus the suprema on the right-hand sides of \eqref{eq:max1s} and \eqref{eq:max1multi} are the same.
\end{proof}

\label{sec:setup}
\subsection{Quadrature}
In this section we introduce the quadrature formulas in full, following most closely the formulation given in Gorin and Marcus \cite{GM}.

Let $A$ and $B$ be diagonal matrices with real entries $a_1,\ldots,a_N$ and $b_1,\ldots,b_N$, 
and in each of the following formulas, let $U$ be a Haar unitary random matrix and let $\Sigma$ be a uniform permutation matrix. 

The first quadrature formula, concerned with the addition of random matrices, states that 
\begin{align} \label{eq:GM}
\mathbf{E} [ \det (z - (A + UBU^{-1}) ) ]  = \mathbf{E} [ \det (z - (A + \Sigma B \Sigma^{-1})) ] .
\end{align}
The second, concerned with multiplication, states that 
\begin{align} \label{eq:GM2}
\mathbf{E} [ \det (z - A \cdot UBU^{-1} ) ]  = \mathbf{E} [ \det (z - A \cdot \Sigma B \Sigma^{-1} ) ] .
\end{align}
The final quadrature formula is concerned with the minors of random matrices. Letting $[A]_k$ denote the $k \times k$ principal minor (i.e.\ top-left corner) of an $N \times N$ matrix, we have 
\begin{align} \label{eq:GM3}
\mathbf{E} [ \det (z -  [UAU^{-1}]_k ) ]  = \mathbf{E}[ \det (z -  [\Sigma A \Sigma^{-1}]_k ) ] ,
\end{align}
where the determinants inside the expectations in \eqref{eq:GM3} are of $k \times k$ matrices. 

We have used a different formulation in \eqref{eq:GM3} to equation (4) of \cite{GM}, which states that $\mathbf{E}[ \det( z - [UAU^{-1}]_k) ] =k!/N! ( \partial/\partial z )^{N-k} \prod_{i=1}^N (z - a_i)$. It is a simple combinatorial exercise to establish the equivalence of the two formulations.

For our purposes we will require multivariate versions of the quadrature formulas, suitable for arbitrary finite sums and products of matrices. We now state the following generalisation of \eqref{eq:GM} and \eqref{eq:GM2}:

\begin{proposition} [Multivariate versions of the quadrature formulas] \label{prop:multiquad}
Let $\Sigma_1,\ldots,\Sigma_d$ and $U_1,\ldots,U_d$ be independent $N \times N$ random matrices, where each $\Sigma_i$ is a uniform permutation matrix, and each $U_i$ is a Haar unitary random matrix, all independent. Given diagonal matrices $A_1,\ldots,A_d$, write
\begin{align*}
A^u_j := U_j A_j U^{-1}_j \qquad \text{and} \qquad A^s_j := \Sigma_j A_j \Sigma_j^{-1}.
\end{align*}
Let $z \in \mathbb{C}$. Then for addition we have
\begin{align} \label{eq:GM0m}
\mathbf{E} [  \det (z - (A_1^u +\cdots+ A_d^u))] = \mathbf{E}[  \det (z - (A_1^s +\cdots+ A_d^s)) ],
\end{align}
for multiplication we have
\begin{align} \label{eq:GM1m} 
\mathbf{E} \left[  \det \left(z - A_1^u \cdots A_d^u \right) \right] = \mathbf{E} \left[  \det \left(z - A_1^s \cdots A_d^s \right) \right], 
\end{align}
and for taking minors we have
\begin{align} \label{eq:GM2m}
\mathbf{E} [  \det (z - [A_1^u]_k )] = \mathbf{E} [  \det (z - [A_1^s]_k )].
\end{align}
\end{proposition}

The proof of Proposition \ref{prop:multiquad} may be found in Appendix \ref{sec:multiquadproof}.

\subsection{Proof set up of Theorem \ref{thm:main}}
Recall the quantile function $T_\mu:[0,1] \to \mathbb{R}$ defined in \eqref{eq:quantile}.
For each $N$, $1 \leq j \leq d$ consider the diagonal matrix
\begin{align}\label{eq:diag}
A_{N,j} := \mathrm{diag}(T_{\mu_j}(i/N) : i=1,\ldots,N).
\end{align}
As $N \to \infty$, the empirical spectrum of $A_{N,j}$ converges weakly to $\mu_j$.

As mentioned in the introduction, our strategy for proving Theorem \ref{thm:mainmulti} is based on an asymptotic analysis of the multivariate quadrature formulas \eqref{eq:GM0m}, \eqref{eq:GM1m} and \eqref{eq:GM2m} under the asymptotic regime \eqref{eq:diag}. Our next result is concerned with the asymptotic behaviour of the expectations over the unitary group occurring in these equations.

\begin{thm} \label{thm:UN}
Let $\mu_1,\ldots,\mu_d$ be probability measures on the real line with compact support, and let $A_{N,j}$ be as in \eqref{eq:diag}. Let $U_{N,j}$ be independent Haar unitary random matrices, and write $A^u_{N,j} := U_{N,j}A_{N,j}U_{N,j}^{-1}$. Then whenever $z > \sum_{j=1}^d E_{\mu_j}^+$ we have
\begin{align} \label{eq:UGM}
\lim_{N \to \infty} \frac{1}{N } \log \mathbf{E} [  \det (z - (A_{N,1}^u +\cdots+ A_{N,d}^u)) ]   =  \int_{-\infty}^\infty \log(z-x) ( \mu_1 \boxplus \cdots \boxplus \mu_d )(\mathrm{d}x) . 
\end{align}
Suppose that we additionally assume that $\mu_2,\ldots,\mu_d$ are supported on $[0,\infty)$. Then whenever $z > \max \prod_{j=1}^d E_{\mu_j}^\pm$ we have
\begin{align} \label{eq:UGM2}
\lim_{N \to \infty} \frac{1}{N } \log  \mathbf{E} [  \det (z - (A_{N,1}^u \cdot \cdots \cdot A_{N,d}^u)) ]= \int_{-\infty}^\infty \log(z-x)   ( \mu_1 \boxtimes \cdots \boxtimes \mu_d )(\mathrm{d}x) . 
\end{align}
Finally, for $z > E^+_{\mu_1}$ we have 
\begin{align} \label{eq:UGM3}
 \lim_{N \to \infty} \frac{1}{\tau N } \log \mathbf{E} [ \det (z -  [A_{N,1}^u]_{\lfloor \tau N \rfloor} ) ]  = \int_{-\infty}^\infty  \log(z-x)  [\mu_1]_\tau (\mathrm{d}x) . 
\end{align}
\end{thm}

Theorem \ref{thm:UN} is proved in Section \ref{sec:Uconv} using results in the finite free probability literature.

Our next result is parallel to Theorem \ref{thm:UN}, but with the expectations instead taken over the symmetric group.

\begin{thm}\label{thm:SN}
Let $\mu_1,\ldots,\mu_d$ be probability measures on the real line with compact support, and let $A_{N,j}$ be as in \eqref{eq:diag}. Let $\Sigma_{N,j}$ be independent uniform permutation matrices, and write $A^s_{N,j} := \Sigma_{N,j}A_{N,j}\Sigma_{N,j}^{-1}$. Recall from \eqref{eq:Hdef} that $H(\gamma)$ is the entropy of a copula. Then whenever $z > \sup_{ s \in [0,1]^d} \sum_{j=1}^d T_{\mu_j}(s_j)$ we have
\begin{align} \label{eq:SGM}
\lim_{N \to \infty} \frac{1}{N } \log \mathbf{E} [  \det (z - (A_{N,1}^s +  \cdots + A_{N,d}^s)) ]  &=\sup_{\gamma} \left\{ \int_{[0,1]^d}  \log \left(z - \sum_{j=1}^d T_{\mu_j}(s_j) \right) \gamma(s)\mathrm{d}s  - H(\gamma) \right\}.
\end{align}
If $z > \sup_{ s \in [0,1]^d} \prod_{j=1}^d T_{\mu_j}(s_j)$ then
\begin{align} \label{eq:SGM2}
\lim_{N \to \infty} \frac{1}{N } \log \mathbf{E} [  \det (z - (A_{N,1}^s \cdot \cdots \cdot A_{N,d}^s)) ]  &=\sup_{\gamma} \left\{ \int_{[0,1]^d}  \log \left(z - \prod_{j=1}^d T_{\mu_j}(s_j) \right) \gamma(s)\mathrm{d}s  - H(\gamma) \right\}.
\end{align}
Whenever $z > \sup_{s \in [0,1]}T_{\mu_1}(s)$ we have 
\begin{align} \label{eq:SGM3}
\lim_{N \to \infty} \frac{1}{\tau N } \log \mathbf{E}[ \det (z -  [A_{N,1}^s]_{\lfloor \tau N \rfloor} ) ]=\frac{1}{\tau}\sup_{\gamma} \left\{ \int_{[0,1]^2} \log(z-T_{\mu_1}(s_1))\mathrm{1}_{\{ s_2 > 1-\tau \}} \gamma(s)\mathrm{d}s - H(\gamma) \right\}.
\end{align}

\end{thm}
Theorem \ref{thm:SN} is proved in Section \ref{sec:Sconv} using tools from large deviation theory.

We remark that \eqref{eq:SGM2} has weaker conditions than its unitary group analogue, \eqref{eq:UGM2}: the former requires only that $\mu_1,\ldots,\mu_d$ have compact support, the latter that $\mu_2,\ldots,\mu_d$ are also supported on $[0,\infty)$.

In any case, we now note that Theorem \ref{thm:main2} (which itself implies Theorem \ref{thm:main} and Theorem \ref{thm:mainmulti}) now follows as a corollary of Theorem \ref{thm:UN}, Theorem \ref{thm:SN}, and the quadrature formulas:

\begin{proof}[Proof of Theorem \ref{thm:main2} assuming Theorem \ref{thm:UN} and Theorem \ref{thm:SN}]
By \eqref{eq:GM0m}, \eqref{eq:GM1m} and \eqref{eq:GM2m}, the respective left-hand sides of \eqref{eq:UGM}, \eqref{eq:UGM2} and \eqref{eq:UGM3} are the same as the respective left-hand sides of \eqref{eq:SGM}, \eqref{eq:SGM2} and \eqref{eq:SGM3}. Comparing the right-hand sides now yields Theorem \ref{thm:main2}.
\end{proof}

Thus in order to complete the proof of our main results, our task in the next two sections is to prove Theorem \ref{thm:UN} and Theorem \ref{thm:SN}.

\section{Convergence of the unitary group side} \label{sec:Uconv}
In this section we prove Theorem \ref{thm:UN}.

\subsection{A preliminary inequality} \label{eq:Uineq}

Before giving a full proof of Theorem \ref{thm:UN}, let us note first that by using Jensen's inequality and the compactness of the supports of $\mu_1,\ldots,\mu_d$, we can use the convergence in \eqref{eq:aconv}, \eqref{eq:mconv} and \eqref{eq:cconv} to prove an inequality holds in the settings of each of \eqref{eq:UGM}, \eqref{eq:UGM2} and \eqref{eq:UGM3}.

For example, in the setting of \eqref{eq:UGM}, writing $\Omega_N$ for the empirical spectrum of $A_{N,1}^u +\cdots+ A_{N,d}^u$, by Jensen's inequality we have 
\begin{align} \label{eq:UGM1pre}
\frac{1}{N } \log \mathbf{E}[   \det (z - (A_{N,1}^u +\cdots+ A_{N,d}^u))  ] &= \frac{1}{N} \log \mathbf{E}\left[ \exp \left\{ N \int_{-\infty}^\infty \log(z-x) \Omega_N(\mathrm{d}x) \right\} \right] \nonumber \\
& \geq  \mathbf{E} \left[ \int_{-\infty}^\infty \log(z-x) \Omega_N(\mathrm{d}x) \right]. 
\end{align} 
Now since $\Omega_N$ converges almost-surely to $ \mu_1 \boxplus \cdots \boxplus \mu_d$ in the weak topology, and the functional $x \mapsto \log(z-x)$ is continuous and uniformly bounded on the possible support of $\Omega_N$, it follows that 
\begin{align} \label{eq:UGMlower}
\lim \inf_{N \to \infty} \frac{1}{N } \log \mathbf{E}[   \det (z - (A_{N,1}^u +\cdots+ A_{N,d}^u))  ] \geq \int_{-\infty}^\infty \log(z-x) ( \mu_1 \boxplus \cdots \boxplus \mu_d )(\mathrm{d}x),
\end{align}
establishing a lower bound in the setting of \eqref{eq:UGM}.

Similar arguments establish lower bounds in the settings of \eqref{eq:UGM2} and \eqref{eq:UGM3}.

We note that a bound analogous to \eqref{eq:UGMlower} can be proved relating the asymptotic expected characteristic polynomial of $A_{N,1}^s +\cdots+ A^s_{N,d}$ with the multiple classical convolution $\mu_1 \ast \cdots \ast \mu_d$. This bound may be used in conjunction with Theorem \ref{thm:UN} (but \emph{without} requiring Theorem \ref{thm:SN}) to give an alternative proof of a multivariate analogue of Corollary \ref{cor:ineq}.

\subsection{Finite free probability}
We now give a complete proof of Theorem \ref{thm:UN} using tools from finite free probability.

Given Hermitian matrices $A_N$ and $B_N$ with real eigenvalues $a_1,\ldots,a_N$ and $b_1,\ldots,b_N$ write
\begin{align} \label{eq:d2d}
\mu_N := \frac{1}{N} \sum_{i=1}^N \delta_{a_i} \qquad \text{and} \qquad \nu_N := \frac{1}{N} \sum_{i=1}^N \delta_{b_i}
\end{align}
for their respective empirical spectra.

Among other things, \textbf{finite free probability} is concerned with three operations $\mu_N \boxplus_N \nu_N$, $\mu_N \boxtimes_N \nu_N$ and $[\mu_N]_k$ on real probability measures of the form \eqref{eq:d2d} (i.e.\ a sum of $N$ real Dirac masses of size $1/N$) that create a new probability measure of this same form. 

We consider first finite additive free convolution $\mu_N \boxplus_N \nu_N$. If we consider the expectation of the characteristic polynomial of $A_N + U_N B_N U_N^{-1}$ where $U_N$ is Haar unitary, we find if we write
\begin{align*}
\Delta(z) := \mathbf{E} [ \det(z - (A_N + U_N B_N U_N^{-1})) ]  = \prod_{i=1}^N (z - c_i),
\end{align*}
then the roots $c_1,\ldots,c_N$ of $\Delta(z)$ are also real \cite{MSS3}.  We now define the finite additive free convolution $\mu_N \boxplus \nu_N$ to be the empirical measure of these roots:
\begin{align*}
\mu_N \boxplus_N \nu_N := \frac{1}{N} \sum_{i=1}^N \delta_{c_i}.
\end{align*}
By considering instead the expected characteristic polynomial of $A_{N,1}^u +\cdots+ A_{N,d}^u$, we can also define the multivariate additive finite free convolution
\begin{align*}
\mu_{N,1} \boxplus_N \cdots \boxplus_N \mu_{N,d} 
\end{align*}
of $d$ probability measures, each a sum of $N$ equal Dirac masses.

Alternatively, again in the setting of \eqref{eq:d2d}, if $B_N$ is positive semidefinite, write $d_1,\ldots,d_N$ for the roots of the expected characteristic polynomial of  $A_N \cdot U_N B_N U_N^{-1}$, that is
\begin{align} \label{eq:multchar}
 \mathbf{E} [ \det(z - A_N \cdot U_N B_N U_N^{-1} ) ]= \prod_{i=1}^N (z - d_i).
\end{align}

Then it turns out that the $d_i$ are also real. As above, we define 
\begin{align*}
\mu_N \boxtimes_N \nu_N := \frac{1}{N} \sum_{i=1}^N \delta_{d_i}
\end{align*}
to be the empirical measure of the $d_i$. We call the operation $\boxtimes_N$ the multiplicative finite free convolution of $\mu_N$ and $\nu_N$. Again, we can define the multivariate multiplicative finite free convolution $\mu_{N,1} \boxtimes_N \cdots \boxtimes_N \mu_{N,d}$ whenever $\mu_{N,2},\ldots,\mu_{N,d}$ are supported on $[0,\infty)$. 

We mention here that since the product of Hermitian matrices need not be Hermitian, in the setting of \eqref{eq:multchar} some authors prefer to consider the Hermitian matrix $B_N^{1/2} U_N A_N U_N^{-1} B_N^{1/2}$, whose eigenvalues have the same distribution as those of $A_NU_NB_NU_N^{-1}$. In order to maintain simple operations and cleaner formulas, particularly in the multivariate case, we eschew this convention.

Finally, let $[U_NA_NU_N^{-1}]_k$ denote the principal $k\times k$ minor of $U_NA_NU_N^{-1}$, where $A_N$ has empirical spectrum $\mu_{N,1}$. Write $e_1,\ldots,e_k$ for the roots of the expected characteristic polynomial of this minor, so that 
\begin{align*}
 \mathbf{E} [ \det(z - [ U_N A_N U_N^{-1}]_k ) ]= \prod_{i=1}^k (z - e_i).
\end{align*}
Again, the $e_i$ are real. Now define a probability measure
\begin{align*}
[\mu_{N,1}]_k := \frac{1}{k} \sum_{i=1}^k \delta_{e_i},
\end{align*}
which we call the finite free compression of $\mu_{N,1}$. 

The following theorem encompasses several results in the literature, characterising how the operations of finite free probability behave asymptotically as the constituent measures involved converge.

\begin{thm}[ \cite{AP}, \cite{AGP}, \cite{HK}, \cite{Mar21}] \label{thm:ff}
For $1 \leq j \leq d$ and $N \geq 1$ let $\mu_{N,j}$ be a sequence of measures of the form \eqref{eq:d2d} satisfying the weak convergence
\begin{align*}
\mu_{N,j} \to \mu_j, 
\end{align*}
where $\mu_j$ are probability measures with compact support. Suppose further that for each $j$, the sequence $(\mu_{N,j})_{N \geq 1}$ is supported on a uniformly bounded subset of the real line. 

Then for finite additive free convolution we have the weak convergence
\begin{align} \label{eq:ffadd}
\mu_{N,1} \boxplus_N \cdots \boxplus_N \mu_{N,d}  \to \mu_{1} \boxplus \cdots \boxplus \mu_{d} .
\end{align}

For finite multiplicative free convolution we have the weak convergence 
\begin{align} \label{eq:ffmult}
\mu_{N,1} \boxtimes_N \cdots \boxtimes_N \mu_{N,d}  \to \mu_{1} \boxtimes \cdots \boxtimes \mu_{d} .
\end{align}

Finally, for finite free compression we have the weak convergence, for $\tau \in (0,1)$,
\begin{align} \label{eq:ffcomp}
[\mu_{N,1}]_{\lfloor \tau N \rfloor} \to [\mu_1]_\tau.
\end{align}
\end{thm} 

We now give explicit references for Theorem \ref{thm:ff}. It is a straightforward exercise by induction to show that if \eqref{eq:ffadd} and \eqref{eq:ffmult} hold for the case $d=2$, then they hold for all $d \geq 2$. Thus it suffices to provide references for the case $d=2$. The first equation, \eqref{eq:ffadd}, is Corollary 5.5 of Arizmendi and Perales \cite{AP}. The next equation, \eqref{eq:ffmult}, is Theorem 1.4 of Arizmendi, Garza-Vargas and Perales \cite{AGP}. We should say that both of these results are implicit in the earlier work of Marcus \cite{Mar21}. The third equation, \eqref{eq:ffcomp}, follows after a brief calculation from the case $d=2$ of the second: in the setting of \eqref{eq:d2d} one can let $A_N$ be the diagonal $N \times N$ matrix whose first $k$ diagonal entries are $1$s, and the remaining $N-k$ are $0$s. See \cite{AGP} for further details on this calculation, as well as work by Hoskins and Kabluchko \cite{HK} and by Steinerberger \cite{steinerberger1, steinerberger2} for further discussion of finite free compression.

Theorem \ref{thm:UN} now follows as a simple consequence of Theorem \ref{thm:ff}:
\begin{proof}[Proof of Theorem \ref{thm:UN}]
We prove \eqref{eq:UGM}; the proofs of \eqref{eq:UGM2} and \eqref{eq:UGM3} are almost identical.

Note from the definition of finite additive free convolution that
\begin{align*}
\frac{1}{N } \log \mathbf{E} [  \det (z - (A_{N,1}^u +\cdots+ A_{N,d}^u)) ]   =  \int_{-\infty}^\infty \log(z-x) ( \mu_{N,1} \boxplus_N \cdots \boxplus_N \mu_{N,d} )(\mathrm{d}x) . 
\end{align*}
Now note that $x \mapsto \log(z-x)$ is continuous and uniformly bounded on the supports of $ \mu_{N,1} \boxplus_N \cdots \boxplus_N \mu_{N,d}$, and thus by the weak convergence of $ \mu_{N,1} \boxplus_N \cdots \boxplus_N \mu_{N,d}$ to $\mu_{1} \boxplus \cdots \boxplus \mu_{d} $, we obtain \eqref{eq:UGM}. 
\end{proof}

\subsection{The large deviations of the empirical spectrum of $A_N + U_NB_NU_N^{-1}$}

We close this section by sketching an alternative approach to proving the $d=2$ case of the first equation \eqref{eq:UGM} of Theorem \ref{thm:UN} using the recent work on the large deviations behaviour of the empirical spectrum of $A_N+U_NB_NU_N^{-1}$ described in \eqref{eq:fern2}. 

The crucial point is that the large deviations of the empirical spectrum $\Omega_N$ of $A_N+ U_NB_NU_N^{-1}$ occur at rate $N^2$, and thus we take as a starting point an upper bound of the form
\begin{align} \label{eq:joker}
\mathbf{P} \left( \Omega_N \notin B_\delta(\mu \boxplus \nu) \right) \leq C_\delta e^{ - c_\delta N^2} \qquad c_\delta,C_\delta > 0,
\end{align}
where $B_\delta(\pi)$ denotes the ball of radius $\delta$ around a probability measure $\pi$ in a suitable metric. The bound \eqref{eq:joker} holds as a consequence of the results in either \cite{BGH} or \cite{NS2}. 

Provided \eqref{eq:joker} holds, we have
\begin{align*}
\mathbf{E} [ \det(z - (A_N + U_N B_N U_N^{-1})) ]  \leq  \exp \left\{ N \sup_{ \omega \in B_\delta(\mu \boxplus \nu) } \int_{-\infty}^\infty \log(z -x ) \omega(\mathrm{d}x) \right\} + C_\delta e^{ - c_\delta N^2} e^{ CN} 
\end{align*}
where we can take $C := \sup_{ x \in [E_\mu^-,E_\mu^+] , y \in [E_\nu^-,E_\nu^+] } |\log (z - (x+y))|$. 

In particular, for any $\delta > 0$, we have 
\begin{align*}
\lim \sup_{N \to \infty} \frac{1}{N} \log\mathbf{E} [ \det(z - (A_N + U_N B_N U_N^{-1})) ]  \leq   \sup_{ \omega \in B_\delta(\mu \boxplus \nu) } \int_{-\infty}^\infty \log(z -x ) \omega(\mathrm{d}x).
\end{align*}

Since $\delta$ is arbitrary, and $\omega \mapsto \int_{-\infty}^\infty \log(z-x) \omega(\mathrm{d}x)$ is a continuous functional on probability measures, we obtain 
\begin{align} \label{eq:UGMupper}
\lim \sup_{N \to \infty} \frac{1}{N} \log\mathbf{E} [ \det(z - (A_N + U_N B_N U_N^{-1})) ]  \leq \int_{-\infty}^\infty \log(z -x ) (\mu \boxplus \nu) (\mathrm{d}x).
\end{align}
Combining \eqref{eq:UGMupper} with the lower bound in \eqref{eq:UGMlower}, we obtain \eqref{eq:UGM}.

\section{Convergence of the symmetric group side} \label{sec:Sconv}

In this section we prove Theorem \ref{thm:SN}. Through this section we will use the notation $C_{z,\bm \mu}$ for a constant depending on $z$ and on a $d$-tuple $\bm \mu$ of probability measures that may change from line to line. 

\subsection{Preliminaries}
Recall from \eqref{eq:quantile} that $T_\mu$ denotes the quantile function associated with $\mu$. 
Recall that we use the notation $\odot$ as a placeholder denoting either addition or multiplication. 
That is, $\odot := +$ or $\odot := \times$. To lighten notation, we assume the operation $\odot$ is always to be performed first, so that in particular
\begin{align*}
z - x_1 \odot \cdots \odot x_d := z - (x_1 \odot \cdots \odot x_d ).
\end{align*}

Let $\bm{\sigma}_N = (\sigma_{N,1},\ldots,\sigma_{N,d})$ be a $d$-tuple of permutations of $\{1,\ldots,N\}$. We associate with $\bm \sigma_N$ two different probability measures on $[0,1]^d$ by setting
\begin{align} \label{eq:gamma}
\gamma_N(\mathrm{d}s) = \gamma^{\bm{\sigma}_N}(\mathrm{d}s)  := N^{d-1} \sum_{i=1}^N \mathrm{1}\left\{ s_j \in \left( \frac{\sigma_{N,j}(i)-1}{N} , \frac{\sigma_{N,j}(i)}{N} \right] ~~ \forall j=1,\ldots,d \right\} \mathrm{d}s,
\end{align}
and
\begin{align} \label{eq:tildegamma}
\tilde{\gamma}_N(\mathrm{d}s) = \tilde\gamma^{\bm{\sigma}_N}(\mathrm{d}s)  := \frac{1}{N} \sum_{i=1}^N \delta_{(\sigma_{N,1}(i)/N,\ldots,\sigma_{N,d}(i)/N)}(\mathrm{d}s).
\end{align}
Note that $\gamma^{\bm \sigma_N}$ is a copula, but $\tilde\gamma^{\bm{\sigma}_N}$ is not. When $\bm \sigma_N$ is understood, we will use the shorthand $\gamma_N$ and $\tilde{\gamma}_N$ for $\gamma^{\bm{\sigma}_N}$ and $\tilde\gamma^{\bm{\sigma}_N}$ respectively.

\subsection{The large deviation principle}
Let $\mathcal{X}$ be a topological space. A rate function on $\mathcal{X}$ is simply a 
lower semicontinuous function $H:\mathcal{X} \to [0,+\infty]$. A sequence $(X_N)_{N \geq 1}$ of random variables taking values in $\mathcal{X}$ is said to satisfy a \textbf{large deviation principle with speed $N$ and rate function $H$} if for all Borel subsets $E$ of $\mathcal{X}$ we have 
\begin{align*}
 - \inf_{x \in E^{\mathrm{o}} } H(x)  \leq  \lim \inf_{N \to \infty} \frac{1}{N} \log \mathbf{P}( X_N \in E ) \leq \lim \sup_{N \to \infty} \frac{1}{N} \log \mathbf{P}( X_N \in E ) \leq  - \inf_{x \in \overline{E}} H(x).
\end{align*}
The rate function $H$ is said to be \textbf{good} if the sets $\{ x: H(x) \leq \alpha \}$ are compact for all $\alpha \geq 0$. See \cite{DZ} for further details on the theory of large deviations.

We will now see that the sequence of random copulas $(\gamma_N)_{N \geq 1}$ associated with a $d$-tuple of independent uniform permutations satisfies a large deviation principle. Namely, following Kenyon et al.\ \cite{KKRW}, write $\Gamma_d$ for the collection of copulas on $[0,1]^d$. This set may be topologised with the metric  
\begin{align} \label{eq:metric}
\mathsf{d}_\Box ( \gamma, \gamma') := \sup_R| \gamma(R) - \gamma'(R) |,
\end{align}
where $R$ ranges over all cylinder sets $R = [a_1,b_1] \times \cdots \times [a_d,b_d]$. Thus $(\Gamma_d,\mathsf{d}_\Box)$ is a metric space.

As noted in \cite{KKRW}, the topology on $\Gamma_d$ generated by the metric in \eqref{eq:metric} is equivalent to the topology of weak convergence. To see this, define the difference of distribution functions $G_{\gamma,\gamma'}:[0,1]^d \to [-1,1]$ by  
\begin{align*}
G_{\gamma,\gamma'}(s_1,\ldots,s_d) := \gamma'( [0,s_1] \times \cdots \times [0,s_d] ) - \gamma( [0,s_1] \times \cdots \times [0,s_d] ).
\end{align*}
Then $\tilde{\mathsf{d}}_\Box(\gamma,\gamma') := \sup_{s \in [0,1]^d} |G_{\gamma,\gamma'}(s)|$ creates a metric which, by definition, generates the topology of weak convergence of probability measures on $[0,1]^d$. Moreover, using the inclusion-exclusion formula and the triangle inequality we have
\begin{align*}
\tilde{\mathsf{d}}_\Box(\gamma,\gamma') \leq \mathsf{d}_\Box ( \gamma, \gamma') \leq 2^d \tilde{\mathsf{d}}_\Box(\gamma,\gamma').
\end{align*}
Thus the two metrics are equivalent, and generate the same topology.

The following is a generalisation of Theorem 1 of \cite{KKRW} (see also \cite{trashorras, wu}). 
\begin{thm}[\cite{KKRW}] \label{thm:KKRW}
For each $N$, let $\gamma_N := \gamma^{\bm \sigma_N}$ be the copula associated with a $d$-tuple $\bm \sigma_N = (\sigma_{N,1},\ldots,\sigma_{N,d})$ of independent uniformly chosen permutations of $\{1,\ldots,N\}$. Then the sequence $(\gamma_N)_{N \geq 1}$ of random elements of $\Gamma_d$ satisfies a large deviation principle with speed $N$ and good rate function $H(\gamma)$ given in \eqref{eq:Hdef}.
\end{thm}

The proof of this result in the case $d=2$ is given in the appendix of \cite{KKRW}. In Appendix \ref{sec:appkkrw} of the present article, we sketch the adaptations necessary to generalise the result to $d\geq 3$.

\subsection{Proof of \eqref{eq:SGM} and \eqref{eq:SGM2}}
We begin in this section by proving the first two equations, \eqref{eq:SGM} and \eqref{eq:SGM2}, of Theorem \ref{thm:SN}. We will see in the next section that the third equation, \eqref{eq:SGM3}, follows as a relatively simple consequence of the second. 

Let $\bm \mu = (\mu_1,\ldots,\mu_d)$ be a $d$-tuple of probability measures with compact support, and let
\begin{align} \label{eq:zb}
z > \sup_{ s \in [0,1]^d} T_{\mu_1}(s_1) \odot \cdots \odot T_{\mu_d}(s_d).
\end{align} 
Define $F_{\bm \mu}^{\odot}:[0,1]^d \to \mathbb{R}$ by
\begin{align} \label{eq:Fappe}
F_{\bm \mu}^{\odot}(s) := \log(z - T_{\mu_1}(s_1) \odot \cdots \odot T_{\mu_d}(s_d)).
\end{align}
Since each $\mu_j$ has compact support and \eqref{eq:zb} holds, $F_{\bm \mu}^{\odot}$ is bounded.

Given a probability measure $\gamma$ on $[0,1]^d$ we now define
\begin{align} \label{eq:Vfunct}
\phi^{\odot}_{\bm \mu}(\gamma ) := \langle F_{\bm \mu}^{\odot} , \gamma \rangle := \int_{[0,1]^d} F_{\bm \mu}^{\odot}(s) \gamma(\mathrm{d}s).
\end{align}
It is useful to think of $\phi^{\odot}_{\bm \mu}:\mathcal{M}_1([0,1]^d) \to \mathbb{R}$ as a functional on the space $\mathcal{M}_1([0,1]^d)$ of probability measures on $[0,1]^d$.
Since $F_{\bm \mu}^{\odot}$ is bounded, $\phi_{\bm \mu}^{\odot}$ is bounded. Both $F_{\bm \mu}^{\odot}$ and $\phi_{\bm \mu}^{ \odot}$ depend on $z$, but we suppress this dependence in the notation. 

The key idea is that we can express the determinants that occur in the equations \eqref{eq:SGM} and \eqref{eq:SGM2} of Theorem \ref{thm:SN} in terms of the random measure $\tilde{\gamma}_N$ defined in \eqref{eq:tildegamma}. Indeed, recall from the statement of Theorem \ref{thm:SN} that each $A_{N,j}$ is a diagonal matrix with diagonal entries $\{ T_{\mu_j}(i/N) : i =1,\ldots,N\}$. If $\Sigma_{N,j}$ is a permutation matrix, then $A_{N,j}^s := \Sigma_{N,j}A_{N,j}\Sigma_{N,j}^{-1}$ is again diagonal, with the entries of $A_{N,j}$ shuffled according to the permutation matrix $\Sigma_{N,j}$. In particular, $A_{N,1}^s \odot \cdots \odot A_{N,d}^s$ is a diagonal matrix (a sum of diagonal matrices if $\odot = +$, and a product of diagonal matrices if $\odot = \times$), and the $i^{\text{th}}$ element on the diagonal is given by $ T_{\mu_1}(\sigma_1(i)/N) \odot \cdots \odot T_{\mu_d}(\sigma_d(i)/N)$. It follows that since $\Sigma_{N,1},\ldots,\Sigma_{N,d}$ are uniform permutation matrices we have
\begin{align} \label{eq:deta1}
\det(z - ( A_{N,1}^s \odot \cdots \odot A_{N,d}^s ) ) = \prod_{i=1}^N (z - T_{\mu_1}(\sigma_{N,1}(i)/N) \odot \cdots \odot T_{\mu_d}(\sigma_{N,d}(i)/N) )= e^{N \langle F_{\bm \mu}^{\odot} , \tilde{\gamma}_N \rangle },
\end{align}
where again we are using the shorthand $\tilde \gamma_N := \tilde \gamma^{\bm \sigma_N}$.

Using \eqref{eq:deta1}, we now package \eqref{eq:SGM} and \eqref{eq:SGM2} in the following statement:

\begin{thm} \label{thm:sym}
Let $\bm \mu := (\mu_1,\ldots,\mu_d)$ be a $d$-tuple of probability measures on the real line with compact support, and let $A_{N,j}$ be as in \eqref{eq:diag}. Let $\tilde \gamma_N := \tilde{\gamma}_{\bm \sigma_N}$ be the measure associated with a $d$-tuple of uniformly chosen permutations. Then whenever $z > \sup_{ s \in [0,1]^d} T_{\mu_1}(s_1) \odot \cdots \odot T_{\mu_d}(s_d)$ we have
\begin{align} \label{eq:sym}
\lim_{N \to \infty} \frac{1}{N} \log \mathbf{E} [ e^{N \langle F_{\bm \mu}^{\odot} , \tilde{\gamma}_N \rangle } ] = \sup_\gamma \left\{\langle F_{\bm \mu}^{\odot} , \gamma \rangle -  H(\gamma) \right\},
\end{align}
where the supremum is taken over all copulas $\gamma$ on $[0,1]^d$.
\end{thm}

Our main tool in proving Theorem \ref{thm:sym} is Varadhan's lemma, for which we give a statement here imposing strong but simple conditions. See \cite[Theorem 4.3.1]{DZ} for a statement with milder conditions.

\begin{thm} [Varadhan's Lemma, Theorem 4.3.1 of \cite{DZ}] \label{thm:varadhan}
Let $(X_N)_{N \geq 1}$ be a sequence of random variables taking values in a topological space $\mathcal{X}$ that satisfy a large deviation principle with speed $N$ and good rate function $H:\mathcal{X} \to [0,\infty]$. Let $\phi:\mathcal{X}\to \mathbb{R}$ be a bounded and continuous function. 
Then 
\begin{align} \label{eq:supremum}
\lim_{ N \to \infty }\frac{1}{N} \log \mathbf{E}[ e^{ N \phi(X_N)} ] = \sup_{x \in \mathcal{X}} \left( \phi(x) - H(x) \right).
\end{align}
\end{thm}

In light of \eqref{eq:deta1}, we would ultimately like to apply Varadhan's lemma to the functional $\phi^{\odot}_{\bm \mu}( \gamma ) := \langle F_{\bm \mu}^{\odot} , \gamma \rangle$ with the sequence $(\tilde{\gamma}_N)_{N \geq 1}$. However, there are several technical issues to consider. First of all, the left-hand side of \eqref{eq:sym} is expressed in terms of $\langle F_{\bm \mu}^{\odot} , \tilde{\gamma}_N \rangle$, rather than $\langle F_{\bm \mu}^{\odot} , \gamma_N \rangle$, and it is $(\gamma_N)_{N \geq 1}$ rather than $(\tilde{\gamma}_N)_{N \geq 1}$ that satisfies the large deviation principle, Theorem \ref{thm:KKRW}.
Second,
it is not a priori obvious that the functional $\gamma \mapsto \langle F_{\bm \mu}^{\odot} , \gamma \rangle$ is continuous for $\gamma \in \Gamma_d$ with the topology generated by the metric \eqref{eq:metric}, since the quantile functions $T_{\mu_j}:[0,1] \to \mathbb{R}$ (and hence $F_{\bm \mu}^{\odot}:[0,1]^d \to \mathbb{R}$) need not be continuous.

With a view to addressing the former issue described above, we have the following lemma, which states that $\langle F_{\bm \mu}^{\odot} , \tilde{\gamma}_N  \rangle $ and $\langle F_{\bm \mu}^{\odot} , \gamma_N \rangle $ are asymptotically similar:

\begin{lemma} \label{lem:cra}
Let $\bm \mu = (\mu_1,\ldots,\mu_d)$ be a $d$-tuple of probability measures with compact support, and let $z > \sup_{ s \in [0,1]^d} T_{\mu_1}(s_1) \odot \cdots \odot T_{\mu_d}(s_d)$. Then there is a constant $C_{z, \bm \mu}$ such that for any $N \geq 1$ and for any $\gamma_N$ and $\tilde{\gamma}_N$ associated with a $d$-tuple $\bm \sigma_N := (\sigma_{N,1},\ldots,\sigma_{N,d})$ of permutations of $\{1,\ldots,N\}$ we have
\begin{align} \label{eq:cra2}
|\langle F_{\bm \mu}^{\odot} , \tilde{\gamma}_N \rangle- \langle F_{\bm \mu}^{\odot} , \gamma_N \rangle| \leq C_{z,\bm \mu}/N.
\end{align}
\end{lemma}
\begin{proof}
A brief calculation using the definitions \eqref{eq:gamma} and \eqref{eq:tildegamma} of $\gamma_N$ and $\tilde{\gamma}_N$ tells us that
\begin{align} \label{eq:warwick1}
|\langle F_{\bm \mu}^{\odot} , \tilde{\gamma}_N \rangle- \langle F_{\bm \mu}^{\odot} , \gamma_N \rangle|  &:= \left| \int_{[0,1]^d} F_{\bm \mu}^{\odot}(s) \tilde{\gamma}_N(\mathrm{d}s) - \int_{[0,1]^d} F_{\bm \mu}^{\odot}(s) \gamma_N(\mathrm{d}s ) \right| \nonumber \\
&\leq \frac{1}{N} \sum_{ i = 1}^N G_N(\sigma_{N,1}(i),\ldots,\sigma_{N,d}(i) ),
\end{align}
where for $1 \leq k_1,\ldots,k_d \leq N$ we have 
\begin{align*}
G_N(k_1,\ldots,k_d) := \sup \left\{ |F_{\bm \mu}^{\odot}(s) - F_{\bm \mu}^{\odot}(t)| :  s,t \in \left( \frac{k_1-1}{N}, \frac{k_1}{N} \right] \times \cdots \times \left( \frac{k_d-1}{N}, \frac{k_d}{N} \right] \right\}.
\end{align*}
We would like to find an upper bound for $G_N(k_1,\ldots,k_d)$ defined in terms of coordinatewise quantities. In this direction, for $1 \leq j \leq d, 1 \leq i \leq N$ define
\begin{align*}
a^{(j)}_i := \sup \left\{  | F_{\bm \mu}^{\odot}(s') - F_{\bm \mu}^{\odot}(s) | : s,s' \in [0,1]^d,  \frac{i-1}{N} < s_j, s_j' \leq \frac{i}{N}, s_k'=s_k \text{ for $k \neq j$}\right\}.
\end{align*}
By the triangle inequality,
\begin{align} \label{eq:warwick2}
G_N(k_1,\ldots,k_d) \leq \sum_{j=1}^d a_{k_j}^{(j)}.
\end{align}

Note that $\log(z-x \odot y)$ is uniformly Lipschitz in the $x$ variable for all $x$ and $y$ varying in a bounded interval such that $z-x \odot y$ is positive and bounded away from zero. Using this observation with $x = T_{\mu_j}(s)$ and $y = T_{\mu_1}(s_1) \odot \cdots \odot T_{\mu_{j-1}}(s_{j-1}) \odot T_{\mu_{j+1}}(s_{j+1}) \odot T_{\mu_d}(s_d)$, we have
\begin{align} \label{eq:jasmine}
a_i^{(j)} \leq C^{(j)}_{z, \bm \mu} (T_{\mu_j}(i/N) - T_{\mu_j}((i-1)/N)),
\end{align}
for some constant $C^{(j)}_{z,\bm \mu}$ not depending on $i$ or $N$. Summing \eqref{eq:jasmine} over $1 \leq i \leq N$, we obtain 
\begin{align} \label{eq:warwick3}
\sum_{i=1}^N a^{(j)}_i \leq C^{(j)}_{z, \bm \mu},
\end{align}
for a possibly different constant $C^{(j)}_{z, \bm \mu}$ not depending on $N$.

Using \eqref{eq:warwick1} in conjunction with \eqref{eq:warwick2} to obtain the first inequality below, using the fact that $\sigma_{N,j}$ is a bijection on $\{1,\ldots,N\}$ and then interchanging the order of summation to obtain the following equality, and then using \eqref{eq:warwick3} to obtain the final inequality, we have
\begin{align*}
|\langle F_{\bm \mu}^{\odot} , \tilde{\gamma}_N \rangle- \langle F_{\bm \mu}^{\odot} , \gamma_N \rangle|  &\leq \frac{1}{N} \sum_{i=1}^N \sum_{j=1}^d a^{(j)}_{\sigma_{N,j}(i)} = \frac{1}{N} \sum_{j=1}^d \sum_{i=1}^N a_i^{(j)} \leq  C_{z, \bm \mu}/N,
\end{align*}
as required.  
\end{proof}

We now turn to considering the continuity properties of the functional $\phi_{\bm \mu}^{\odot}:\Gamma_d \to \mathbb{R}$ given in \eqref{eq:Vfunct}. We prove the following:

\begin{lemma} \label{lem:functcon}
Suppose $\mu_1,\ldots,\mu_d$ have compact support. 
Let $z > \sup_{s \in [0,1]^d} (T_{\mu_1}(s_1) \odot \cdots \odot T_{\mu_d}(s_d) )$. Then the functional $\phi_{\bm \mu}^{\odot} : \Gamma_d \to \mathbb{R}$ on the space of permutons given by \eqref{eq:Vfunct} is continuous in the topology generated by the metric \eqref{eq:metric}.
\end{lemma}

In the course of the proof of Lemma \ref{lem:functcon}, we will use the following standard consequence of the portmanteau theorem for weak convergence.

\begin{lemma} \label{lem:techa}
If $(\mu_k)_{k \geq 1}$ is a sequence of probability measures on a metric space $E$ that converges weakly to a measure $\mu$, and $f:E \to \mathbb{R}$ is a bounded function whose discontinuity set $D_f := \{ x \in E: f \text{ is not continuous at $x$} \}$ satisfies $\mu(D_f) = 0$, then $\int f \mathrm{d}\mu_k$ converges to $\int f \mathrm{d}\mu$ as $k \to \infty$. 
\end{lemma}

Lemma \ref{lem:techa} is proved in Appendix \ref{sec:techa}.

\begin{proof}[Proof of Lemma \ref{lem:functcon}]
We begin by studying the discontinuity set $D$ of the function $F_{\bm \mu}^{\odot}:[0,1]^d \to \mathbb{R}$, which need not be continuous. Each function $T_{\mu_i}:[0,1] \to \mathbb{R}$ is monotone increasing. Accordingly, each $T_{\mu_i}$ has a countable discontinuity set $D_i \subseteq [0,1]$. Define
\begin{align*}
D_i' := \{ s \in [0,1]^d : s_i \in D_i \}.
\end{align*}
Since $x \mapsto \log(z-x)$ is continuous and bounded for all $x$ in compact subintervals of $(-\infty, z)$, it follows that $F_{\bm \mu}^{\odot}:[0,1]^d \to \mathbb{R}$ is also bounded, and its set of discontinuities $D$ satisfies the inclusion 
\begin{align*}
D \subseteq D_1' \cup \cdots \cup D_d'.
\end{align*}
Let $\gamma$ be a permuton, and let $\mathrm{Leb}$ denote Lebesgue measure on the (one-dimensional) unit interval $[0,1]$. Then using the fact that $\gamma$ projects as Lebesgue measure onto each axis, we have
\begin{align*}
\gamma(D) \leq \sum_{i=1}^d \gamma(D_i') = \sum_{i=1}^d \mathrm{Leb}(D_i) = 0,
\end{align*}
where the final equality above follows from the fact that the Lebesgue measure of a countable set is zero. 

Suppose now that $(\gamma_n)_{n \geq 1}$ is a sequence of permutons that converges to $\gamma$ in the topology generated by the metric \eqref{eq:metric}; equivalently, $(\gamma_n)_{n \geq 1}$ converge weakly to $\gamma$. 
Since our function $F_{\bm \mu}^{\odot}$ has a discontinuity set $D \subseteq [0,1]^d$ for which $\gamma(D) = 0$, it follows from Lemma \ref{lem:techa} that if $(\gamma_k)_{k \geq 1}$ is any sequence of permutons (or probability measures, for that matter) on $[0,1]^d$ that converges to $\gamma$ in the topology generated by \eqref{eq:metric}, then $\int F_{\bm \mu}^{\odot} \mathrm{d}\gamma_k \to \int F_{\bm \mu}^{\odot} \mathrm{d}\gamma$. Thus $\phi_{\bm \mu}^{\odot} : \Gamma_d \to \mathbb{R}$ is continuous on the space of permutons, completing the proof.
\end{proof}

We are now equipped to prove Theorem \ref{thm:sym}:

\begin{proof}[Proof of Theorem \ref{thm:sym}]
By virtue of Lemma \ref{lem:cra}, we have
\begin{align} \label{eq:Q1}
\frac{1}{N} \log \mathbf{E}\left[ e^{ N \langle F_{\bm \mu}^{\odot} , \tilde{\gamma}_N \rangle } \right] = \frac{1}{N} \log \mathbf{E} \left[ e^{ N \langle F_{\bm \mu}^{\odot}, \gamma_N \rangle } \right] + O_{z,\bm \mu}(1/N).
\end{align}
On the other hand, the functional $\phi_{\bm \mu}^{\odot} : \Gamma_d \to \mathbb{R}$ on the space of permutons given by $\phi_{\bm \mu}^{\odot}(\gamma) := \int_{ [0,1]^d} F_{\bm \mu}^{\odot}(s) \gamma(\mathrm{d}s)$ is bounded, and by Lemma \ref{lem:functcon}, continuous. Since, according to Theorem \ref{thm:KKRW}, the sequence $(\gamma_N)_{N \geq 1}$ satisfies a large deviation principle with speed $N$ and good rate function $H(\gamma)$, it follows from Varadhan's lemma, Theorem \ref{thm:varadhan}, that 
\begin{align} \label{eq:Q2}
\lim_{N \to \infty} \frac{1}{N} \log \mathbf{E}\left[ e^{ N \langle F_{\bm \mu}^{\odot} , \gamma_N \rangle } \right] = \sup_{ \gamma} \left\{ \phi_{\bm \mu}^{\odot} (\gamma) - H(\gamma) \right\}. 
\end{align}
Combining \eqref{eq:Q1} and \eqref{eq:Q2} we obtain Theorem \ref{thm:sym}.
\end{proof}

\subsection{Proof of \eqref{eq:SGM3}}
We reiterate explicitly that by \eqref{eq:deta1}, Theorem \ref{thm:sym}, which we just proved, implies the first two equations \eqref{eq:SGM} and \eqref{eq:SGM2} of Theorem \ref{thm:SN}. To complete the proof of Theorem \ref{thm:SN}, it remains to prove the final equation, \eqref{eq:SGM3}. It turns out that this final equation can be proved fairly easily by manipulating the multiplicative case of Theorem \ref{thm:sym} with $d=2$. 

\begin{thm} \label{thm:sym2}
Let $\mu_1$ have compact support, and $z > \sup_{s \in [0,1]} T_{\mu_1}(s)$. Let $A_{N,1}$ be as in \eqref{eq:diag} and recall the random matrix $A_{N,1}^s := \Sigma_{N,1}A_{N,1} \Sigma_{N,1}^{-1}$, where $\Sigma_{N,1}$ is a uniform random $N \times N$ permutation matrix. Then  
\begin{align} \label{eq:SGM3b}
& \lim_{N \to \infty} \frac{1}{\tau N } \log \mathbf{E} \left[ \det (z -  [A_{N,1}^s]_{\lfloor \tau N \rfloor} ) \right] =\frac{1}{\tau}\sup_{\gamma} \left\{ \int_{[0,1]^2} \log(z-T_{\mu_1}(s_1))\mathrm{1}_{\{ s_2 > 1-\tau\} } \gamma(\mathrm{d}s) - H(\gamma) \right\},
\end{align}
where the supremum is taken over all copulas on $[0,1]^2$.
\end{thm}
\begin{proof}
Let $a_i = T_{\mu_1}(i/N)$. Note that 
\begin{align*}
\mathbf{E}\left[ \det(z - [A_{N,1}^s]_{\lfloor \tau N \rfloor} ) \right] = \mathbf{E} \left[ \prod_{i = 1}^{ \lfloor \tau N \rfloor } ( z - a_{\sigma_{N,1}(i)} ) \right]= z^{ - (N -\lfloor \tau N \rfloor  ) }  \mathbf{E} \left[\prod_{i=1}^{N } ( z -  a_{\sigma_{N,1}(i)} b_{\sigma_{N,2}(i)}) \right],
\end{align*}
where $b_i := \mathrm{1}_{\{ \frac{i}{N} > 1 - \tau \}}$, and $\sigma_{N,1}$ and $\sigma_{N,2}$ are independent uniform random permutations in $\mathcal{S}_N$. In particular, we have the equality in expectation
\begin{align} \label{eq:glazunov}
\mathbf{E}\left[ \det(z - [A_{N,1}^s]_{\lfloor \tau N \rfloor}) \right] =  z^{ - (N -\lfloor \tau N \rfloor  ) }  \mathbf{E} [\det ( z -  A_{N,1}^s A_{N,2}^s)]
\end{align}
where $A_{N,2}^s = \Sigma_{N,2}A_{N,2} \Sigma_{N,2}^{-1}$, for a uniform permutation matrix $\Sigma_{N,2}$, and $A_{N,2}$ is a matrix with $N - \lfloor \tau N \rfloor$ zeroes and $\lfloor \tau N \rfloor$ ones along the diagonal. That is, $A_{N,2} = \mathrm{diag}( T_{\mu_2}(i/N) : i = 1,\ldots,N)$ where  
\begin{align*}
T_{\mu_2}(s_2) = \mathrm{1}_{\{ s_2 > 1 -\tau \}}, \qquad \text{or equivalently, }\qquad \mu_2(\mathrm{d}x) = (1 - \tau)\delta_0(\mathrm{d}x)  + \tau \delta_1(\mathrm{d}x).
\end{align*}
Since $\mu_2$ also has compact support, the case $d=2, \odot = \times$ of Theorem \ref{thm:sym} applies, and consequently taking $\frac{1}{N} \log$ of both sides of \eqref{eq:glazunov} and sending $N$ to infinity we obtain
\begin{align} \label{eq:carmel}
\lim_{N \to \infty} \frac{1}{N} \log \mathbf{E}\left[ \det(z - [A_{N,1}^s]_{\lfloor \tau N \rfloor} ) \right] = - (1-\tau)\log z + \sup_\gamma \{ \langle F_{\mu_1,\mu_2}^\times , \gamma \rangle - H(\gamma) \},
\end{align}
where the supremum is taken over all copulas on $[0,1]^2$.

To obtain \eqref{eq:SGM3b} as written, note $T_{\mu_1}(s_1)T_{\mu_2}(s_2)=0$ for $(s_1,s_2)$ in $[0,1] \times [0,1-\tau)$. Since $\gamma$ is a copula, it gives mass $1-\tau$ to the region $[0,1] \times [0,1-\tau)$, so that 
\begin{align} \label{eq:carmel2}
\langle F_{\mu_1,\mu_2}^\times , \gamma \rangle  = \int_{[0,1]^2} \log ( z - T_{\mu_1}(s_1)T_{\mu_2}(s_2)) \gamma(\mathrm{d}s) = (1 - \tau)\log z + \int_{[0,1] \times [1-\tau,1]} \log(z-T_{\mu_1}(s_1)) \gamma(\mathrm{d}s).
\end{align}
Combining \eqref{eq:carmel} and \eqref{eq:carmel2} we thus obtain
\begin{align} \label{eq:carmel3}
\lim_{N \to \infty} \frac{1}{\tau N} \log \mathbf{E}[ \det(z - [A_{N,1}^s]_{\lfloor \tau N \rfloor}) ] = \frac{1}{\tau}\sup_\gamma \left\{ \int_{[0,1]^2} \log(z-T_{\mu_1}(s_1))\mathrm{1}_{\{s_2 > 1-\tau \} } \gamma(\mathrm{d}s)   - H(\gamma) \right\},
\end{align}
completing the proof.
\end{proof}

That completes the proof of Theorem \ref{thm:SN}: by \eqref{eq:deta1}, the equations \eqref{eq:SGM} and \eqref{eq:SGM2} are proved in Theorem \ref{thm:sym}, and Theorem \ref{thm:sym2} is the proof of equation \eqref{eq:SGM3}. 

\section{Variational calculations} \label{sec:variational} 

\subsection{Explicit solutions to the variational problems}
In this section, we show how Lemma \ref{lem:bgn} may be used to solve the variational problems set out in Theorem \ref{thm:main} and its multivariate extension, Theorem \ref{thm:mainmulti}. 

With the symbol $\boxdot$ representing either free additive convolution $\boxplus$ or free multiplicative convolution $\boxtimes$, and $\odot$ representing the associated addition or multiplication of real-valued random variables, Theorem \ref{thm:mainmulti} says that for sufficiently large real $z$ we have
\begin{align} \label{eq:newmulti}
\int_{-\infty}^\infty \log(z-x) (\mu_1 \boxdot \cdots \boxdot \mu_d) (\mathrm{d}x) = \sup_\Pi \left\{ \mathbf{E}_\Pi \left[ \log \left( z - X_1 \odot \cdots \odot X_d \right) \right] - H(\Pi \mid \otimes_{j=1}^d \mu_j) \right\}.
\end{align}
Recall that the compression case of Theorem \ref{thm:main} reads 
\begin{align} \label{eq:max3again}
\tau \int_{-\infty}^\infty  \log(z-x)  [\mu]_\tau (\mathrm{d}x) = \sup_{\Pi} \left\{ \mathbf{E}_\Pi [ \log (z - X)\mathrm{1}_{Y > 1-\tau} ] - H(\Pi \mid \mu \otimes \nu_{\mathrm{uni}}) \right\},
\end{align}
where $\nu_{\mathrm{uni}}$ denotes the Lebesgue measure on $[0,1]$, and the supremum is taken over all couplings $\Pi$ of $\mu$ and $\nu_{\mathrm{uni}}$.

Our next result calculates the quantity in \eqref{eq:newmulti}, and in the process describes the coupling that attains the supremum. The resulting equation \eqref{eq:apollo2} and the subsequent systems of equations --- \eqref{eq:add1}, \eqref{eq:add2} and \eqref{eq:add3} in the additive case, and \eqref{eq:mult1}, \eqref{eq:mult2} and \eqref{eq:mult3} in the multiplicative case --- are variations on well-known subordination results in the literature \cite{BB, BZ, Bia, Vent1, Voi02}.

\begin{thm}\label{thm:explicitnew}
With the symbol $\boxdot$ representing either free additive convolution $\boxplus$ or free multiplicative convolution $\boxtimes$, the coupling $\Pi_{*,z}$ achieving the supremum in \eqref{eq:newmulti} is the unique coupling whose Radon--Nikodym derivative with respect to product measure takes the form
\begin{align} \label{eq:RN3}
\frac{\mathrm{d} \Pi_{*,z} }{ \mathrm{d}(\mu_1 \otimes \cdots \otimes \mu_d) } = \omega(z) \frac{  z - X_1 \odot \cdots \odot X_d }{ \prod_{i=1}^d (\omega_i(z) - X_i) }.
\end{align}
With $\omega(z),\omega_1(z),\ldots,\omega_d(z)$ as above, the value of the supremum in \eqref{eq:newmulti} is given by 
\begin{align} \label{eq:apollo2}
\int_{-\infty}^\infty \log(z-x) (\mu_1 \boxdot \cdots \boxdot \mu_d ) (\mathrm{d}x) = - \log\omega(z) + \sum_{i=1}^d \int_{-\infty}^\infty \log(\omega_i(z) - x) \mu_i (\mathrm{d}x).
\end{align}
The unique choices of $\omega(z),\omega_1(z),\ldots,\omega_d(z)$ such that \eqref{eq:RN3} defines a coupling of $\mu_1,\ldots,\mu_d$ are encoded in the
following systems of equations. In the case $\boxdot = \boxplus$, $\omega(z),\omega_1(z),\ldots,\omega_d(z)$ satisfy
\begin{align}
\sum_{i=1}^d \omega_i(z) &=  z + \frac{d-1}{c(z)}, \label{eq:add1} \\
c(z) &= G_i(\omega_i(z)) \qquad \text{for each $1 \leq i \leq d$}, \label{eq:add2}\\
 \omega(z) &= c(z)^{-(d-1)}, \label{eq:add3}
\end{align}
and in the case $\boxdot = \boxtimes$, they instead satisfy
\begin{align}
\prod_{i=1}^d \omega_i(z) &= z \left( 1 + \frac{1}{c(z)} \right)^{d-1},\label{eq:mult1}\\
1+ c(z) &= \omega_i(z) G_i(\omega_i(z)) \qquad \text{for each $1 \leq i \leq d$},\label{eq:mult2}\\ 
\omega(z) &= c(z)^{-(d-1)}. \label{eq:mult3}
\end{align} 
In \eqref{eq:add2} and \eqref{eq:mult2}, $G_i(\omega) := \int_{-\infty}^\infty (\omega - x)^{-1}\mu_i(\mathrm{d}x)$ denotes the Cauchy transform of $\mu_i$. 
\end{thm}

On occasion in the sequel, to lighten notation we will omit the $z$-dependency in our notation, writing $\omega_i := \omega_i(z), \omega := \omega(z)$ and $c := c(z)$. 

We also note that \eqref{eq:RN3} is identical to \eqref{eq:RN}, stated in Section \ref{sec:further}.
\begin{proof}
By Proposition 2.2 of \cite{BGN} with $c(x_1,\ldots,x_d) = \log(z - x_1 \odot \cdots \odot x_d)$, there is a unique coupling attaining the supremum in \eqref{eq:newmulti}, and it is the unique coupling of $\mu_1,\ldots,\mu_d$ whose Radon--Nikodym derivative with respect to product measure takes the form
\begin{align} \label{eq:RN2}
\frac{\mathrm{d} \Pi_{*,z} }{ \mathrm{d}(\mu_1 \otimes \cdots \otimes \mu_d) } = ( z - X_1 \odot \cdots \odot X_d )\prod_{i=1}^d A_i(X_i),
\end{align}
for some real-valued functions $A_1,\ldots,A_d:\mathbb{R} \to \mathbb{R}$.
We now show that the $A_i(x)$ are necessarily scalar multiples of $(\omega_i - x)^{-1}$. Without loss of generality, we consider $i = 1$. Indeed, since $\Pi_{*,z}$ is a coupling of the random variables $X_1,\ldots,X_d$, for all bounded and measurable $\phi:\mathbb{R} \to \mathbb{R}$, we must necessarily have
\begin{align*}
\mathbf{E}_{\mu_1 \otimes \cdots \otimes \mu_d} \left[ \phi(X_1)  ( z - X_1 \odot \cdots \odot X_d )\prod_{i=1}^d A_i(X_i) \right] = \mathbf{E}_{\mu_1} [ \phi(X_1) ],
\end{align*}
which, in particular, implies that for $\mu_1$-a.e.\ $x \in \mathbb{R}$ we have 
\begin{align} \label{eq:leafer} 
A_1(x) \mathbf{E}_{\mu_1 \otimes \cdots \otimes \mu_d} \left[ ( z - x \odot X_2 \odot \cdots \odot X_d )\prod_{i=2}^d A_i(X_i) \right] = 1.
\end{align}
Since (in both cases $\boxdot = \boxplus$ and $\boxdot = \boxtimes$) the term inside the expectation is of the form $a + bx$, it follows that $A_1(x)$ is the reciprocal of such a function. By symmetry, the same is true of $A_2(x),\ldots,A_d(x)$. It follows that the coupling attaining the supremum in \eqref{eq:newmulti} necessarily takes the form set out in \eqref{eq:RN3}, 
for some $\omega,\omega_1,\ldots,\omega_d$ depending on $z$.

With \eqref{eq:RN3} now established, we show that the equation \eqref{eq:apollo2} follows from a brief calculation using \eqref{eq:RN3} in the right-hand side of \eqref{eq:newmulti}. Indeed, setting $\Pi = \Pi_{*,z}$ in the quantity on the right-hand side of \eqref{eq:newmulti}, and then subsequently using the definition \eqref{eq:kl} of relative entropy, we have 
\begin{align*}
 &\mathbf{E}_{\Pi_{*,z}} \left[ \log \left( z - X_1 \odot \cdots \odot X_d \right) \right] - H(\Pi_{*,z} \mid \otimes_{j=1}^d \mu_j)\\
 &= \mathbf{E}_{\Pi_{*,z}} \left[ \log \left( z - X_1 \odot \cdots \odot X_d \right)  - \log \left( \omega \frac{  z - X_1 \odot \cdots \odot X_d }{ \prod_{i=1}^d (\omega_i - X_i) } \right)  \right]\\
 &= - \log \omega + \sum_{i=1}^d \mathbf{E}_{\Pi_{*,z}} \left[ \log (\omega_i - X_i)  \right] = - \log \omega + \sum_{i=1}^d \mathbf{E}_{\mu_i} [ \log (\omega_i - X_i) ],
\end{align*}
where the final equality above follows from the fact that the marginal law of $X_i$ under $\Pi_{*,z}$ is simply $\mu_i$. That proves \eqref{eq:apollo2}.

It remains to characterise the unique choices $\omega$, $\omega_1,\ldots,\omega_d$ making $\Pi_{*,z}$ a coupling.
Appealing to our newfound knowledge that the $A_i$ take the form set out in \eqref{eq:RN3}, after dividing both sides by $A_1(x_1)$, the relation \eqref{eq:leafer} reads 
\begin{align} \label{eq:leafer2}
\omega_1 - x = \omega \mathbf{E}_{\mu_1 \otimes \cdots \otimes \mu_d} \left[ \frac{ z - x \odot X_2 \odot \cdots \odot X_d }{ \prod_{i=2}^d (\omega_i - X_i) } \right],
\end{align}
for all $x$ in the support of $\mu_1$. Provided $\mu_1$ is nondegenerate, we can relate the constant coefficient and the coefficient of $x$ in \eqref{eq:leafer2}. 

It is at this point that we have to distinguish between the cases $\boxdot = \boxplus$ and $\boxdot = \boxtimes$. First, in the setting $\boxdot = \boxplus$, we establish \eqref{eq:add1}, \eqref{eq:add2} and \eqref{eq:add3}. When $\boxdot = \boxplus$, so that $\odot = +$, relating the constant coefficient and the coefficient of $x$ in \eqref{eq:leafer2}, we obtain the equations
\begin{align}
\omega_1 &= \omega \mathbf{E}_{\mu_1 \otimes \cdots \otimes \mu_d}\left[ \frac{ z - (X_2 + \cdots + X_d) }{ \prod_{i =2}^d (\omega_i - X_i ) } \right],\label{eq:n1} \\
1 &= \omega \mathbf{E}_{\mu_1 \otimes \cdots \otimes \mu_d} \left[ \frac{1}{\prod_{i=2}^d (\omega_i - X_i) } \right]. \label{eq:n2}
\end{align}
By the independence of $X_2,\ldots,X_d$, the latter equation, \eqref{eq:n2}, reads $1 = \omega \prod_{i=2}^d G_i(\omega_i)$. By symmetry, for each $1 \leq j \leq d$ we have 
\begin{align} \label{eq:n3}
1 = \omega \prod_{1 \leq i \leq d, i \neq j} G_i(\omega_i).
\end{align}
In particular, $G_i(\omega_i)$ does not depend on $1 \leq i \leq d$. Writing $c(z)$ for the common value of $G_i(\omega_i)$, by \eqref{eq:n3} we have 
\begin{align} \label{eq:n4}
G_i(\omega_i) = c(z) \quad \text{and} \quad \omega = c(z)^{-(d-1)},
\end{align}
which is precisely the content of \eqref{eq:add2} and \eqref{eq:add3}. In the case $\boxdot = \boxplus$, it remains to prove \eqref{eq:add1}. Now after rearranging, \eqref{eq:n1} states that
\begin{align} \label{eq:n5}
\frac{\omega_1}{\omega} &= \mathbf{E}_{\mu_1 \otimes \cdots \otimes \mu_d} \left[ \frac{ z - \sum_{j=2}^d \omega_j  }{ \prod_{i=2}^d (\omega_i - X_i ) } \right] + \sum_{j=2}^d \mathbf{E}_{\mu_1 \otimes \cdots \otimes \mu_d} \left[ \frac{ \omega_j - X_j }{ \prod_{i=2}^d (\omega_i - X_i ) } \right] \nonumber \\
 &= (z - \sum_{j = 2}^d \omega_j ) c(z)^{d-1} + (d-1)c(z)^{d-2},
 \end{align}
 where the latter equality above follows from the independence of $X_2,\ldots,X_d$ together with \eqref{eq:n4}. The equation \eqref{eq:add1} follows from multiplying both sides of \eqref{eq:n5} by $\omega = c(z)^{-(d-1)}$ and rearranging.

\begin{align}
\omega_1 &= \omega z \mathbf{E}_{\mu_1 \otimes \cdots \otimes \mu_d}\left[ \frac{1}{\prod_{i=2}^d (\omega_i - X_i ) } \right], \label{eq:o1} \\
1 &=  \omega \mathbf{E}_{\mu_1 \otimes \cdots \otimes \mu_d}\left[ \prod_{i=2}^d \frac{X_i}{\omega_i - X_i}  \right].\label{eq:o2}
\end{align}
Again using the independence of $X_2,\ldots,X_d$, and the fact that $\mathbf{E}[X_i(\omega_i-X_i)^{-1}] = -1 + \omega_i G_i(\omega_i)$, the latter equation \eqref{eq:o2} reads $1 = \omega \prod_{i=2}^d ( \omega_i G_i(\omega_i) - 1)$. By symmetry, for each $1 \leq j \leq d$ we have 
\begin{align} \label{eq:neuer}
1 = \omega \prod_{1 \leq i \leq d, i \neq j} (\omega_i G_i(\omega_i) - 1).
\end{align}
In particular, $\omega_i G_i(\omega_i) - 1$ does not depend on $1 \leq i \leq d$. Writing $c(z)$ for the common value of this quantity, by \eqref{eq:neuer} we have
\begin{align} \label{eq:neuer2} 
\omega_i G_i(\omega_i) = c(z) +1 \quad \text{and} \quad \omega = c(z)^{-(d-1)},
\end{align}
which is precisely the content of \eqref{eq:mult2} and \eqref{eq:mult3}. With \eqref{eq:neuer2} at hand, using independence, \eqref{eq:o1} now reads
\begin{align} \label{eq:o11}
\omega_1 = \omega z \prod_{i=2}^d G_i(\omega_i) = \omega z (\omega_2 \cdots \omega_d)^{-1} (c(z)+1)^{d-1}.
\end{align}
Using \eqref{eq:neuer2} in \eqref{eq:o11} and rearranging, we obtain \eqref{eq:mult1}, completing the proof. 
\end{proof}

Our next result is the analogue of the previous result but in the setting of free compression (instead of free additive or multiplicative convolution). The equations \eqref{eq:apollo3} and \eqref{eq:comp1} are variants on known results in this setting, see e.g. \cite{BN}.

\begin{thm} \label{thm:explicit2}
Let $\tau \in (0,1)$. 
The coupling $\Pi_{*,z,\tau}$ achieving the supremum in \eqref{eq:max3again} has Radon--Nikodym derivative against product measure taking the form
\begin{align} \label{eq:RN4}
\frac{\mathrm{d} \Pi_{*,z,\tau} }{ \mathrm{d}(\mu \otimes \nu_{\mathrm{uni}}) } = \frac{ \frac{1}{\tau} \mathrm{1}\{ Y > 1-\tau \} + \frac{\omega-z}{1-\tau} \mathrm{1} \{ Y \leq 1-\tau\} }{ \omega - X}.
\end{align}
Moreover, we have 
\begin{align} \label{eq:apollo3}
\int_{-\infty}^\infty \log(z-x)  [\mu]_\tau (\mathrm{d}x) = \frac{1}{\tau} \int_{-\infty}^\infty \log(\omega-x)\mu(\mathrm{d}x)  + \log \tau - \frac{1-\tau}{\tau} \log \frac{\omega-z}{1-\tau}.
\end{align}
Here $\omega$ is the function of $z$ (and $\tau$) given by the unique solution to the equation
\begin{align} \label{eq:comp1}
1 = \frac{\omega-z}{1-\tau} G_\mu(\omega ).
\end{align}
\end{thm}
\begin{proof}[Proof of Theorem \ref{thm:explicit2}]
Combining the final equation of Theorem \ref{thm:main} with Lemma \ref{lem:bgn}, as well as \eqref{eq:ecom}, we see that 
\begin{align} \label{eq:max1bbb}
\tau \int_{-\infty}^\infty \log(z-x) [\mu]_\tau (\mathrm{d}x) = - \mathbf{E}_{\mu}[\log A_{z,\tau}(X)] - \mathbf{E}_{\nu_{\mathrm{uni}}}[\log B_{z,\tau}(Y)],
\end{align}
where $A_{z,\tau}:\mathbb{R} \to \mathbb{R}$ and $B_{z,\tau}:[0,1] \to \mathbb{R}$ are the unique functions depending on $z$ and on $\tau$ such that $\Pi_{*,z,\tau}$ defined by 
\begin{align} \label{eq:proko}
\frac{ \mathrm{d} \Pi_{*,z,\tau}}{ \mathrm{d}(\mu \otimes \nu_{\mathrm{uni}})} = A_{z,\tau}(X)B_{z,\tau}(Y) (z-X)^{\mathrm{1}_{\{Y > 1- \tau\}} }
\end{align}
is a coupling of the probability measures $\mu$ and $\nu_{\mathrm{uni}}$. 

Like in the proof of Theorem \ref{thm:explicitnew}, in order for $\Pi_{*,z,\tau}$ to be a coupling of $\mu$ and the uniform law $\nu_{\mathrm{uni}}$ on $[0,1]$ we must have both
\begin{align} \label{eq:lobster}
1 = A_{z,\tau}(x) \mathbf{E}_{\nu_{\mathrm{uni}}} \left[ B_{z,\tau}(Y) (z-x)^{\mathrm{1}_{\{Y > 1- \tau\}} } \right] \qquad \text{for $\mu$-almost every $x \in \mathbb{R}$}
\end{align}
and
\begin{align} \label{eq:lobster3}
1 = B_{z,\tau}(y) \mathbf{E}_\mu\left[ A_{z,\tau}(X) (z-X)^{\mathrm{1}_{\{y > 1- \tau\}}} \right] \qquad \text{for $\mathrm{Leb}$-almost every $y \in [0,1]$}.
\end{align}
This forces $A_{z,\tau}(x) = \frac{1}{c_1-c_2x}$ and $B_{z,\tau}(y) = c_3\mathrm{1}_{y \leq 1-\tau } + c_4 \mathrm{1}_{y > 1 - \tau } $ for some $c_1,c_2,c_3,c_4$ depending on $z$ and $\tau$, so that amalgamating $c_1,c_2,c_3,c_4$ we have
\begin{align} \label{eq:lobster2}
A_{z,\tau}(X)B_{z,\tau}(Y) = \frac{C_1 \mathrm{1}_{Y > 1 - \tau }  + C_2 \mathrm{1}_{ Y \leq 1- \tau} }{ \omega - X}
\end{align}
for some functions $C_1,C_2,\omega$ of $z$ and $\tau$. By \eqref{eq:max1bbb} we then have 
\begin{align} \label{eq:max1bbbb}
\tau \int_{-\infty}^\infty \log(z-x) [\mu]_\tau (\mathrm{d}x) &= \mathbf{E}_\mu [ \log(  \omega - X) ] - \mathbf{E}_{\nu_{\mathrm{uni}}}[ \log \left( C_1 \mathrm{1}_{Y > 1 - \tau }  + C_2 \mathrm{1}_{ Y \leq 1- \tau}  \right) ] \nonumber
\\&= \int_{-\infty}^\infty \log(\omega-x)\mu(\mathrm{d}x)  - \tau \log C_1 - (1-\tau) \log C_2.
\end{align}
We now find $C_1$ and $C_2$. Using \eqref{eq:lobster2} in \eqref{eq:lobster}, we obtain $\omega - x = C_1 \tau (z-x) + C_2 (1-\tau)$ for $\mu$-a.e. $x \in \mathbb{R}$, which (provided $\mu$ is not concentrated on an atom) implies
\begin{align} \label{eq:eye}
C_1 = \frac{1}{\tau} \qquad \text{and} \qquad C_2 = \frac{ \omega - z}{1- \tau}.
\end{align}
Combining \eqref{eq:lobster2} and \eqref{eq:eye} we obtain \eqref{eq:RN4}. 
Plugging \eqref{eq:eye} into \eqref{eq:max1bbbb} we obtain \eqref{eq:apollo3}. 

It remains to show that $\omega$ satisfies \eqref{eq:comp1}. 
In this direction, plugging \eqref{eq:eye} into \eqref{eq:lobster2} and then subsequently using \eqref{eq:lobster3}, we obtain the system of equations
\begin{align} \label{eq:lobster4}
1 = \frac{1}{\tau} \mathbf{E}_\mu \left[ \frac{z-X}{\omega-X} \right] \qquad \text{and} \qquad  1 = \frac{\omega-z}{1-\tau} \mathbf{E}_\mu \left[ \frac{1}{\omega-X} \right],
\end{align}
the first of which follows from setting $y > 1-\tau$ in \eqref{eq:lobster3}, and the latter of which follows from setting $y \leq 1-\tau$ in \eqref{eq:lobster3}. Both equations in \eqref{eq:lobster4} are equivalent to one another, and imply \eqref{eq:comp1}.
\end{proof}

\subsection{Cauchy transforms at optimal couplings}
In this section we prove Proposition \ref{prop:cauchynew}. To prepare the proofs of the first two equations, \eqref{eq:max1bnew} and \eqref{eq:max2bnew}, we begin by noting that from \eqref{eq:RN3} that
\begin{align} \label{eq:base0}
\mathbf{E}_{\Pi_{*,z}} \left[ \frac{1}{z - X_1 \odot \cdots \odot X_d } \right] = \mathbf{E}_{\mu_1 \otimes \cdots \otimes \mu_d} \left[  \frac{\omega}{\prod_{i=1}^d (\omega_i - X_i) } \right] = \omega \prod_{i=1}^d G_i(\omega_i).
\end{align}
On the other hand, differentiating the equation \eqref{eq:apollo2} with respect to $z$ we have
\begin{align} \label{eq:base1}
\int_{-\infty}^\infty \frac{1}{z-x} (\mu_1 \boxdot \cdots \boxdot \mu_d)(\mathrm{d}x) = - \frac{\omega'}{\omega} + \sum_{i=1}^d \omega_i' G_i(\omega_i).
\end{align} 
The equations \eqref{eq:max1bnew} and \eqref{eq:max2bnew} amount to showing that the right-hand sides of \eqref{eq:base0} and \eqref{eq:base1} coincide in the respective cases $\boxdot = \boxplus$ and $\boxdot = \boxtimes$. First we consider the former case:

\begin{proof}[Proof of \eqref{eq:max1bnew}]
Suppose $\boxdot = \boxplus$. On the one hand, using \eqref{eq:add2} and \eqref{eq:add3} in \eqref{eq:base0} we have
\begin{align} \label{eq:base01}
\mathbf{E}_{\Pi_{*,z}} \left[ \frac{1}{z - ( X_1 + \cdots + X_d ) } \right] = c(z).
\end{align}
On the other hand, using \eqref{eq:add2} in \eqref{eq:base1} to obtain the first equality below, then differentiating \eqref{eq:add1} to obtain the second, we have
\begin{align} \label{eq:base11}
\int_{-\infty}^\infty \frac{1}{z-x} (\mu_1 \boxplus \cdots \boxplus \mu_d)(\mathrm{d}x) = - \frac{\omega'}{\omega} + c(z) \sum_{i=1}^d \omega_i' = - \frac{\omega'}{\omega} + c(z) \frac{\mathrm{d}}{\mathrm{d}z} \left( z + \frac{d-1}{c(z)} \right).
\end{align} 
Using $\omega = c(z)^{-(d-1)}$ we have $\omega'/\omega = - (d-1)c'/c$, and rearranging, we see that
\begin{align} \label{eq:base111}
\int_{-\infty}^\infty \frac{1}{z-x} (\mu_1 \boxplus \cdots \boxplus \mu_d)(\mathrm{d}x) = c(z).
\end{align} 
Relating \eqref{eq:base01} and \eqref{eq:base111} completes the proof of \eqref{eq:max1bnew}.
\end{proof}

Now we consider the case $\boxdot = \boxtimes$:

\begin{proof}[Proof of \eqref{eq:max2bnew}]
Suppose $\boxdot = \boxtimes$. On the one hand, using \eqref{eq:mult2} in \eqref{eq:base0} to obtain the first equality below, then using \eqref{eq:mult1} and \eqref{eq:mult3} to obtain the second, we have 
\begin{align} \label{eq:base02}
\mathbf{E}_{\Pi_{*,z}} \left[ \frac{1}{z - ( X_1 \cdots X_d ) } \right] = \omega \prod_{i=1}^d \frac{1 + c(z)}{ \omega_i } = \frac{c(z)+1}{z}.
\end{align}
On the other hand, using \eqref{eq:mult2} in \eqref{eq:base1} to obtain the first equality below, we have 
\begin{align} \label{eq:base12}
\int_{-\infty}^\infty \frac{1}{z-x} (\mu_1 \boxtimes \cdots \boxtimes \mu_d)(\mathrm{d}x) = - \frac{\omega'}{\omega} + \sum_{i=1}^d \omega_i' \frac{1+c(z)}{\omega_i} = - \frac{\omega'}{\omega} + (1 + c(z))\frac{\mathrm{d}}{\mathrm{d}z} \log \prod_{i=1}^d \omega_i. 
\end{align} 
Using \eqref{eq:mult1} and reorganising, we see that 
\begin{align} \label{eq:base121}
\int_{-\infty}^\infty \frac{1}{z-x} (\mu_1 \boxtimes \cdots \boxtimes \mu_d)(\mathrm{d}x) = \frac{c(z)+1}{z}.
\end{align}
Relating \eqref{eq:base02} and \eqref{eq:base121} completes the proof of \eqref{eq:max2bnew}.
\end{proof}

To complete the proof of Proposition \ref{prop:cauchynew}, it remains to prove \eqref{eq:max3b}. 

\begin{proof}[Proof of \eqref{eq:max3b}]
On the one hand, differentiating \eqref{eq:apollo3} with respect to $z$ we obtain 
\begin{align} \label{eq:apollo4}
G_{[\mu]_\tau}(z) = \frac{\omega'}{\tau} G_\mu(\omega) - \frac{1-\tau}{\tau} \frac{\omega'-1}{\omega-z}.
\end{align}
Using \eqref{eq:comp1} in \eqref{eq:apollo4}, we obtain
\begin{align} \label{eq:compo}
G_{[\mu]_\tau}(z) = \frac{1-\tau}{\tau} \frac{1}{\omega - z}.
\end{align}
On the other hand, using \eqref{eq:RN4} to obtain the first equality below,
the fact that $X$ and $Y$ are independent with laws $\mu$ and $\nu_{\mathrm{uni}}$ under $\mathbf{E}_{\mu \otimes \nu_{\mathrm{uni}}}$ to obtain the second, and \eqref{eq:comp1} to obtain the third, we have 
\begin{align} \label{eq:proko77}
 \mathbf{E}_{\Pi_{*,z,\tau}} \left[ \frac{1}{z-X} \mathrm{1}_{\{Y > 1-\tau\}} \right] 
&= \frac{1}{\tau} \mathbf{E}_{\mu \otimes \nu_{\mathrm{uni}}} \left[  \frac{1}{\omega-X} \mathrm{1}_{\{Y > 1-\tau\}} \right] = G_\mu(\omega) = \frac{1-\tau}{\omega-z}.
\end{align}
Relating \eqref{eq:compo} and \eqref{eq:proko77}, we obtain \eqref{eq:max3b}.
\end{proof}

\section{Further calculations in free probability} \label{sec:free}

For the sake of completeness, in this section we show how the usual formulas for free convolutions and free compression can be derived from scratch using Theorems \ref{thm:explicitnew} and \ref{thm:explicit2}. We stress that all of the results in Section \ref{sec:71} through \ref{sec:compr} are well-known in the literature. Recall through this section that $G_\mu(z) := \int_{-\infty}^\infty \frac{1}{z-x}\mu(\mathrm{d}x)$ denotes the Cauchy transform. As preparation for our work in this section, we home in on the case $d=2$.
When $\mu_1 = \mu$ and $\mu_2 = \nu$, let us write
\begin{align*}
\omega_\mu := \omega_1 \quad \text{and} \quad \omega_\nu := \omega_2,
\end{align*}
where $\omega_1,\omega_2$ are as in the statement of Theorem \ref{thm:explicitnew}. Here, the $d=2$ case of the system \eqref{eq:add1}, \eqref{eq:add2} and \eqref{eq:add3} reads 
\begin{align}
\omega_\mu + \omega_\nu &= \omega+ z, \label{eq:add1old} \\
 \frac{1}{\omega} &= G_\mu(\omega_\mu) = G_\nu(\omega_\nu),  \label{eq:add2old}
\end{align}
and the $d=2$ case of the system \eqref{eq:mult1}, \eqref{eq:mult2} and \eqref{eq:mult3} reads 
\begin{align}
\omega_\mu \omega_\nu &= z(\omega+1),\label{eq:mult2old}\\
1+ \frac{1}{\omega} &= \omega_\mu G_\mu(\omega_\mu) = \omega_\nu G_\nu(\omega_\nu) . \label{eq:mult1old}
\end{align}
In each case, $\omega(z) = c(z)^{-(d-1)}$ coincides here with $1/c(z)$.

Note that \eqref{eq:base111} reads
\begin{align} \label{eq:nattan0}
G_{\mu \boxplus \nu } (z) = 1/\omega,
\end{align}
and \eqref{eq:base121} reads
\begin{align} \label{eq:nattan1}
G_{\mu \boxtimes \nu } (z) = \frac{\omega+1}{z \omega }.
\end{align}
\subsection{$R$-transform}\label{sec:71}
Given a probability measure $\mu$ and its associated Cauchy transform $G_\mu(z)$, we define its $R$-transform $R_\mu$ implicitly through the relation
\begin{align} \label{eq:Rdef}
z = \frac{1}{G_\mu(z)} + R_\mu(G_\mu(z)).
\end{align} 

We now give a formal proof of the following well-known consequence of our results.

\begin{thm}[\cite{Vcon1}]
We have 
\begin{align*}
R_{\mu \boxplus \nu}(s) = R_\mu(s) + R_\nu(s).
\end{align*}
\end{thm}
\begin{proof}
Substituting $\omega_\mu$ for $z$ in \eqref{eq:Rdef} we have $\omega_\mu = 1/G_\mu(\omega_\mu) + R_\mu(G_\mu(\omega_\mu))$. Now using \eqref{eq:add2old} we obtain 
\begin{align*}
\omega_\mu = \omega + R_\mu(1/\omega).
\end{align*}
By finding an equivalent equation for $\nu$ rather than $\mu$, and adding the equations together, we obtain
\begin{align*}
\omega_\mu +\omega_\nu = 2 \omega + R_\mu(1/\omega) + R_\nu(1/\omega).
\end{align*}
Using \eqref{eq:add1old} this reduces to
\begin{align*}
z =  \omega + R_\mu(1/\omega) + R_\nu(1/\omega).
\end{align*}
Finally, by \eqref{eq:nattan0} we obtain
\begin{align*}
z =  \frac{1}{G_{\mu \boxplus \nu}(z)}  + R_\mu(G_{\mu \boxplus \nu}(z)) + R_\nu(G_{\mu \boxplus \nu}(z)).
\end{align*}
Using \eqref{eq:Rdef}, and setting $s = G_{\mu \boxplus \nu}(z)$, we see that $R_\mu(s) + R_\nu(s) = R_{\mu \boxplus \nu}(s)$, completing the proof.
\end{proof}

\subsection{$S$-transform}
Now we consider free multiplicative convolution. For simplicity, we assume that the involved measures have non-zero mean. The zero-mean case may be handled using the methods in \cite{RSp}.

Given a probability measure $\mu$, for sufficiently small $z$ define 
\begin{align*}
\psi_\mu(z) := \int_{-\infty}^\infty \sum_{n \geq 1} (zx)^n \mu(\mathrm{d}x).
\end{align*}
It is easily verified that
\begin{align*}
\psi_\mu(z) = \frac{1}{z} G_\mu(1/z)-1.
\end{align*}
Let $\chi_\mu(z)$ denote the inverse function of $\psi_\mu(z)$. The \textbf{$S$-transform} of the measure $\mu$ is the function
\begin{align} \label{eq:Sdef}
S_\mu(z) := \frac{1+z}{z} \chi_\mu(z).
\end{align}
We now give a formal proof of the following standard result relating the $S$-transform of the multiplicative free convolution of measures $\mu$ and $\nu$ to the respective $S$-transforms of $\mu$ and $\nu$.

\begin{thm}[\cite{Vcon2}]
We have
\begin{align} \label{eq:Smult}
S_{\mu \boxtimes \nu}(z) = S_\mu(z) S_\nu(z).
\end{align}
\end{thm}
\begin{proof}
For probability measures $\mu$, set $J_\mu(z) := zG_\mu(z)-1$. Then $J_\mu(z) = \psi_\mu(1/z)$. Expressed in terms of $J_\mu$, \eqref{eq:nattan1} reads as saying
\begin{align} \label{eq:Jeq}
J_{\mu \boxtimes \nu}(z) = 1/\omega,
\end{align}
where, rephrasing \eqref{eq:mult2old} and \eqref{eq:mult1old} with $J_\mu$ in place of $G_\mu$, we have 
\begin{align}
J_\mu (\omega_\mu)&=J_\nu(\omega_\nu) = 1/\omega \label{eq:mult11}\\
\omega_\mu\omega_\nu&= z(\omega+1)\label{eq:mult12}.
\end{align}
We now compare the quantities $S_\mu (1/\omega) S_\nu(1/\omega)$ and $S_{\mu \boxtimes \nu}(1/\omega)$, and show that they coincide. Considering the first of these quantities, by the definition \eqref{eq:Sdef} we have
\begin{align} \label{eq:mozart}
S_\mu(1/\omega)S_\nu(1/\omega)= (\omega+1)^2 \chi_\mu (1/\omega) \chi_\nu(1/\omega).
\end{align}
Plugging \eqref{eq:mult11} into \eqref{eq:mozart} to obtain the first equality below we have
\begin{align} \label{eq:mozart2}
S_\mu(1/\omega)S_\nu(1/\omega) = (\omega+1)^2 \chi_\mu \left(J_\mu(\omega_\mu) \right)\chi_\nu\left(J_\nu(\omega_\nu) \right) = (\omega+1)^2/\omega_\mu\omega_\nu,
\end{align}
where to obtain the second equality above, we have used the fact that since $J_\mu(z) = \psi_\mu(1/z)$, and $\chi_\mu$ is the inverse function of $\psi_\mu$, we have $\chi_\mu(J_\mu(z)) = 1/z$. 

Using \eqref{eq:mult12}, \eqref{eq:mozart2} subsequently reduces to
\begin{align} \label{eq:mozart3}
S_\mu(1/\omega)S_\nu(1/\omega) = (\omega+1)/z.
\end{align}

We turn to considering $S_{\mu \boxtimes \nu}(1/\omega)$. Using the definition \eqref{eq:Sdef} of the $S$-transform to obtain the first equality below, \eqref{eq:Jeq} to obtain the second, and then again using $\chi_\mu(J_\mu(z)) = 1/z$ to obtain the third, we obtain
\begin{align} \label{eq:mozart4}
S_{\mu \boxtimes \nu}(1/\omega) = (\omega+1)\chi_{\mu \boxtimes \nu}( 1/\omega)  = (\omega+1)\chi_{\mu \boxtimes \nu}\left( J_{\mu \boxtimes \nu}(z) \right) =  (\omega+1)/z.
\end{align}
Comparing \eqref{eq:mozart3} and \eqref{eq:mozart4}, we see that \eqref{eq:Smult} holds with $1/\omega$ in place of $z$. Since $\omega$ varies continuously with $z$, it follows that \eqref{eq:Smult} holds, completing the proof.
\end{proof}

\subsection{Relationship between free additive convolution and compression} 
\label{sec:compr}
Recall the $R$-transform defined in \eqref{eq:Rdef}. We now give a formal proof of the following standard result concerning the $R$-transform of the compression of a measure.

\begin{thm}[\cite{NS}]
For $\tau \in (0,1]$ we have
\begin{align} \label{eq:symph3}
R_{[\mu]_\tau}(s) = R_\mu(\tau s).
\end{align}
\end{thm}

\begin{proof}
Noting that $[\mu]_1 = \mu$, throughout this proof we will use $G_\tau$ and $G_1$ as shorthand for $G_{[\mu]_\tau}$ and $G_\mu$, and do similarly with $R_\tau$ and $R_1$. By \eqref{eq:compo}, 
\begin{align} \label{eq:miral}
G_\tau(z) = \frac{1}{\tau} \frac{1 - \tau}{\omega - z},
\end{align}
where, as in \eqref{eq:comp1}, $\omega$ is the solution to the equation $G_1(\omega) = \frac{1-\tau}{\omega-z}$. 

Now, on the one hand, according to the definition of the $R$-transform of $[\mu]_\tau$ we have 
\begin{align} \label{eq:miral2}
R_\tau(G_\tau(z))=z-\frac{1}{G_\tau(z)}.
\end{align}
Using \eqref{eq:miral}, \eqref{eq:miral2} reads
\begin{align} \label{eq:symph}
R_\tau \left( \frac{1}{\tau} \frac{1-\tau}{\omega-z} \right) = ( z - \tau \omega)/(1-\tau).
\end{align}
On the other hand, substituting $\omega$ for $z$ in the definition of the $R$-transform of $\mu$, we have
\begin{align} \label{eq:miral3}
R_1 ( G_1( \omega )) = \omega - \frac{1}{G_1(\omega)}.
\end{align}
Using $G_1(\omega) = (1-\tau)/(\omega-z)$, \eqref{eq:miral3} reads
\begin{align} \label{eq:symph2}
R_1 \left( \frac{1-\tau}{\omega-z} \right) = ( z - \tau \omega)/(1-\tau).
\end{align}
Comparing \eqref{eq:symph} and \eqref{eq:symph2}, and substituting $s =  \frac{1}{\tau} \frac{1-\tau}{\omega-z} $ we obtain \eqref{eq:symph3}.
\end{proof}

Given a probability measure $\mu$, write $\lambda_* \mu$ for the pushforward measure under the map $x \mapsto \lambda x$. In other words, if $X$ has law $\mu$, $\lambda X$ has law $\lambda_*\mu$. It is easily verified that $G_{\lambda_* \mu}(z) = \frac{1}{\lambda}G_\mu(z/\lambda)$, and subsequently that $R_{\lambda_* \mu}(s) = \lambda R_\mu(\lambda s)$. It follows in particular that
\begin{align*}
R_{[\mu]_\tau}(s) = \frac{1}{\tau} R_{\tau_* \mu}(s).
\end{align*}
Recall now that $R_{\mu \boxplus \nu}(s) = R_\mu(s) +R_\nu(s)$. It follows that for integers $k \geq 1$ we have $R_{\mu^{\boxplus k}}(s) = kR_\mu(s)$, where $\mu^{\boxplus k} := \mu \boxplus \cdots \boxplus \mu$ denotes the $k$-fold additive free convolution of $\mu$ with itself. Thus, setting $\tau = 1/k$, it follows that
\begin{align*}
R_{[\mu]_{1/k}}(s) = k R_{\tau_* \mu}(s) = R_{ (\tau_* \mu)^{\boxplus k} } (s) = R_{ \tau_* \mu^{\boxplus k}  } (s).
\end{align*}
By the uniqueness of $R$-transforms, we have arrived at the following known result (see e.g. \cite{ST}).

\begin{thm}
For integers $k\geq 1$, we have the following identity in law relating free compression with additive free convolution
\begin{align} \label{eq:compadd}
[\mu]_{1/k} = (1/k)_* \mu^{\boxplus k} .
\end{align}
\end{thm}

With the $[\mu]_\tau$ defined for all $\tau \in (0,1]$, one can then use \eqref{eq:compadd} to define the fractional additive free convolution, valid for all real $k \geq 1$, by setting $\mu^{ \boxplus k} := k_* [\mu]_{1/k}$; see e.g.\ \cite{BV} or \cite{NS}. For real $k \geq 1$, it is also natural to consider the rescaling $\sqrt{k}_* [\mu]_{1/k} = (1/\sqrt{k})_* \mu^{\boxplus k}$, as for any $k \geq 1$ this measure has the same variance as $\mu$. As mentioned in the introduction, Shlyakhtenko and Tao \cite{ST} recently showed that the free entropy $\chi((1/\sqrt{k})_* \mu^{\boxplus k})$ is monotone increasing in $k$.

\appendix
\section{The equation relating the semicircle law with the Gaussian law } \label{sec:solver}
In this section we prove that if $G$ is a standard Gaussian random variable and, for $z \in \mathbb{C}^+$, $W_z(G)$ is the unique random variable of the form $W_z(G) = c(z) (z - G)e^{f(z)G}$ with $f(z) \in \mathbb{C}^-$ satisfying the conditions $\mathbf{E}[W_z(G)] = 1$ and $\mathbf{E}[G W_z(G)] = 0$, then
\begin{align} \label{eq:solverapp}
\mathbf{E} \left[ \frac{W_z(G)}{z-G} \right] = \frac{z - \sqrt{z^2-4}}{2}.
\end{align}
To prove \eqref{eq:solverapp}, we begin by using the Gaussian moment identities $\mathbf{E}[e^{f(z)G}]=e^{f(z)^2/2}$, $\mathbf{E}[Ge^{f(z)G}]=f(z) e^{f(z)^2/2}$ and $\mathbf{E}[G^2 e^{f(z)G}]=(1+f(z)^2)e^{f(z)^2/2}$ to obtain 
\[
\mathbf{E}[W_z(G)]
= c(z)(z-f(z))e^{f(z)^2/2},\qquad
\mathbf{E}[G W_z(G)]
= c(z)\bigl(zf(z)-(1+f(z)^2)\bigr)e^{f(z)^2/2}.
\]
Hence the conditions reduce to $zf-(1+f^2)=0$ and $c(z)(z-f)e^{f^2/2}=1$. Solving $f^2 - zf + 1 = 0$ and choosing the branch with $f(z) \in \mathbb{C}^-$ for $z \in \mathbb{C}^+$ we obtain 
\begin{align*}
f(z)=\frac{z-\sqrt{z^2-4}}{2} \quad \text{and} \quad c(z)=\frac{e^{-f(z)^2/2}}{z-f(z)} = f(z)e^{-f(z)^2/2}.
\end{align*}
Finally,
\begin{align*}
\mathbf{E}\!\left[\frac{W_z(G)}{z-G} \right]
= c(z)\mathbf{E}[e^{f(z)G}]
= c(z)e^{f(z)^2/2}
= f(z),
\end{align*}
which proves \eqref{eq:solverapp}.

\section{Counterexample to modulus equation inside the support} \label{sec:counterexample}

In this section, we show by way of example that equality is violated in the setting of \eqref{eq:max14}. 

Take $\mu = \nu=\mu_{\mathrm{bern}}:= \frac{1}{2} \delta_{-1} + \frac{1}{2} \delta_1$ as the standard Bernoulli law. Then, on the one hand, since the classical convolution of $\mu_{\mathrm{bern}}$ with itself is $\mu_{\mathrm{bern}} \ast \mu_{\mathrm{bern}} = \frac{1}{4} \delta_{-2}+ \frac{1}{2} \delta_0+\frac{1}{4}\delta_2$ we have 
\begin{align*}
\int_{-\infty}^\infty \log|z-x| (\mu_{\mathrm{bern}} \ast \mu_{\mathrm{bern}}) (\mathrm{d}x) = \frac{1}{4} \log |z-2| + \frac{1}{2} \log |z| + \frac{1}{4} \log|z+2|.
\end{align*}
In particular, setting $z = 1$ we see that $\int_{-\infty}^\infty \log|1-x| ( \mu_{\mathrm{bern}} \ast \mu_{\mathrm{bern}}) (\mathrm{d}x)  = \frac{1}{4}\log 3 > 0$. 

On the other hand, according to \cite[Example 12.8]{NS}, we have $\mu_{\mathrm{bern}} \boxplus \mu_{\mathrm{bern}} = \mu_{\mathrm{arc}}$, where
\begin{align} \label{eq:arc}
\mu_{\mathrm{arc}} (\mathrm{d}x) =\frac{ \mathrm{1}_{x \in [-2,2]} \mathrm{d}x}{ \pi \sqrt{ 4-x^2} }
\end{align}
According to  \cite[Example 3.5, pp.\ 45]{STo}, we have 
\begin{align*}
\int_{-\infty}^\infty \log|z-x|  \mu_{\mathrm{arc}}(\mathrm{d}x)  = 
\begin{cases}
0, \qquad &\text{if $z \in [-2,2]$},\\
\log(\frac{1}{2}|z+\sqrt{z^2-4}|) \qquad &\text{otherwise}.
\end{cases}
\end{align*}
Hence, when $z=1$, we have  $\int_{-\infty}^\infty \log|z-x| (\mu_{\mathrm{bern}} \ast \mu_{\mathrm{bern}}) (\mathrm{d}x)  = \frac{1}{4}\log 3 > 0 = \int_{-\infty}^\infty \log|z-x| (\mu_{\mathrm{bern}} \boxplus \mu_{\mathrm{bern}}) (\mathrm{d}x)$, thereby establishing that setting $\Pi$ as the product measure $\Pi = \mu_{\mathrm{bern}} \otimes \mu_{\mathrm{bern}}$, we violate equality in the setting of \eqref{eq:max14}.

\section{Proof of Proposition \ref{prop:multiquad}} \label{sec:multiquadproof}

\begin{proof}[Proof of Proposition \ref{prop:multiquad}]
First we prove \eqref{eq:GM0m}. 
Given a word $w = (w_1,\ldots,w_d) \in \{u,s\}^d$, (i.e. $w_j \in \{u,s\}$ for each $j$), write $F(w) := \mathbf{E}[  \det (z - \sum_{j=1}^d A^{w_j}_j ) ]$. 
Suppose that for some $1 \leq j \leq d$, $w_j = u$. 
Then using the shorthand $B := \sum_{k \neq j, 1 \leq k \leq d } A_k^{w_k}$, note that $B$ is independent of $U_j$ and $\Sigma_j$. Although \eqref{eq:GM} was stated for diagonal $B$, the same identity holds for arbitrary $B$, as follows immediately by expanding the determinant in principal minors. Conditioning on $B$ and applying this extension of \eqref{eq:GM}, we have
\begin{align*}
F(w) := \mathbf{E}\left[ \det \left( z - B - U_jA_j U_j^{-1} \right) \right] = \mathbf{E}\left[ \det \left( z - B - \Sigma_jA_j \Sigma_j^{-1} \right) \right] = F(w'),
\end{align*}
where $w' = (w'_1,\ldots,w'_d)$ has $w'_i = w_i$ for $i \neq j$, and $w'_j = s$. It follows that $F(w)$ does not depend on the word $w$. In particular, $F(u,\ldots,u) = F(s,\ldots,s)$, proving \eqref{eq:GM0m}.

As for \eqref{eq:GM1m}, for $0 \leq j \leq d$ let
\begin{align*}
G_j := \mathbf{E}[ \det (z - A_1^s \cdots A_j^s A_{j+1}^u \cdots A_d^u ) ].
\end{align*}
If $j < d$, let $A := A_1^s \cdots A_j^s$ and write $A_{j+1}^u \cdots A_d^u = U_{j+1}B U_{j+1}^{-1}$, so that $B = A_{j+1}U_{j+1}^{-1} A_{j+2}^u \cdots A_d^u  U_{j+1}$. We may write $B = A_{j+1} \tilde{A}_{j+2}^u \cdots \tilde{A}_d^u$, where for $j+2 \leq p \leq d$ we have $\tilde{A}_p^u = V_p A_p V_p^{-1}$ with $V_p :=  U_{j+1}^{-1}U_p$. 
By standard properties of Haar measure, each $V_p$ is Haar unitary, and the
$V_{j+2},\ldots,V_d$ are independent of one another and of $U_{j+1}$. In
particular, $B$ is independent of $U_{j+1}$. Again, although \eqref{eq:GM2} was
stated for diagonal $B$, the same identity holds for arbitrary $B$. Conditioning on
$B$ and applying this extension of \eqref{eq:GM2}, we have
\begin{align} \label{eq:laplace}
G_j := \mathbf{E}\left[ \det (z - A_1^s \cdots A_j^s A_{j+1}^u \cdots A_d^u ) \right] = \mathbf{E}[\det(z-A U_{j+1}B U_{j+1}^{-1}) ] = \mathbf{E}[\det(z-A\Sigma_{j+1}B\Sigma_{j+1}^{-1})].
\end{align}
Write $\Sigma_{j+1}B\Sigma_{j+1}^{-1} = A_{j+1}^sC$, where $C = \Sigma_{j+1}\tilde{A}_{j+2}^u \cdots \tilde{A}_d^u \Sigma_{j+1}^{-1}$. Letting $\hat{A}_p^u := \Sigma_{j+1} \tilde{A}_p^u \Sigma_{j+1}^{-1}$, we see that $\hat{A}_{j+2}^u,\ldots,\hat{A}_d^u$ are independent of one another and of $A_1^s,\ldots,A_{j+1}^s$, so that the right-hand side of \eqref{eq:laplace} is equal to $G_{j+1}$. In particular, $G_j = G_{j+1}$, so that by induction, $G_0 = G_d$, proving \eqref{eq:GM1m}.

The final equation, \eqref{eq:GM2m}, is simply \eqref{eq:GM3} in the new notation. 
\end{proof}

\section{Outline of Theorem \ref{thm:KKRW} proof} \label{sec:appkkrw}

The proof of Theorem \ref{thm:KKRW} in the general $d \geq 2$ case is almost identical to that of the $d = 2$ case, which is given in the appendix of \cite{KKRW}. In this brief appendix we outline the key computational aspect of the proof for general $d \geq 2$. 

The pivotal step we provide here is an entropy calculation estimating the number of $d$-tuples of permutations with a certain property. This is equivalent to evaluating the probability that the random measure $\gamma_N$ defined in \eqref{eq:gamma} takes certain values when evaluated on a mesh of the unit cube $[0,1]^d$. In this direction, let $m \geq 1$ be an integer, and for the sake of simplicity we will consider the case where $N = N'm$ is an integer multiple of $m$. Let $[m] = \{1,\ldots,m\}$. Consider now the decomposition
\begin{align*}
(0,1]^d = \bigcup_{\mathbf{r} \in [m]^d } Q_{\bm r} ~\text{ where   $Q_{\bm r} := Q_{r_1,\ldots,r_d} := \left(\frac{r_1-1}{m},\frac{r_1}{m}\right] \times \ldots \times \left(\frac{r_d-1}{m},\frac{r_d}{m}\right]$}.
\end{align*}
Given a copula $\gamma$, write
\begin{align*}
\bar{\gamma}_{\bm r} := m^d \int_{Q_{\bm r}} \gamma(\mathrm{d}s)
\end{align*}
for the average of $\gamma$ on $Q_{\bm r}$. In particular, if $\gamma^{\bm \sigma_N}$ is the copula associated with a $d$-tuple $\bm \sigma_N = (\sigma_{N,1},\ldots,\sigma_{N,d})$ of permutations of $\{1,\ldots,N\}$, define the integer
\begin{align*}
N_{\bm r}^{\bm \sigma} :=  Nm^{-d} \gamma^{\bm \sigma}_{\bm r} = \# \left\{ 1 \leq i \leq N : \forall 1 \leq j \leq d, \frac{(r_j - 1)N}{m} < \sigma_{N,j}(i) \leq \frac{ r_j N}{m} \right\}.
\end{align*}
We have the following lemma, whose proof is a combinatorial exercise which we leave to the reader.
\begin{lemma} \label{lem:comb}
Let $\gamma$ be a fixed copula that is constant on each $Q_{\bm r}$, and write $\gamma_r$ for the value of $\gamma$ on $Q_{\bm r}$. Suppose each $N_{\bm r} := Nm^{-d}\gamma_{\bm r}$ is an integer. Then 
\begin{align*}
\# \{  \bm \sigma_N := (\sigma_{N,1},\ldots,\sigma_{N,d}) \in \mathcal{S}_N^d : N_{\bm r}^{\bm \sigma} = N_{\bm r} ~ \forall \bm r \in [m]^d  \} = \frac{ N! ((N/m)!)^{md} }{ \prod_{ \mathbf{r} \in [m]^d } N_{\bm r}! } .
\end{align*}
\end{lemma}

Note from Lemma \ref{lem:comb}, given a copula $\gamma_r$ such that each $N_{\bm r} :=  Nm^{-d}\gamma_{\bm r}$ is an integer we have
\begin{align*}
\frac{1}{N} \log \mathbf{P}\left(
\bar{\gamma}^{\bm \sigma_N}_{\bm r}=\gamma_{\bm r}
\text{ for all }\bm r\in[m]^d
\right) = \frac{1}{N} \log \left( \frac{1}{N!^d} \frac{ N! ((N/m)!)^{md} }{ \prod_{ \mathbf{r} \in [m]^d } N_{\bm r}! }  \right).
\end{align*}
Using Stirling's formula we have
\begin{align*}
 \log \left( \frac{1}{N!^d} \frac{ N! ((N/m)!)^{md} }{ \prod_{ \mathbf{r} \in [m]^d } N_{\bm r}! }  \right) &= - (d-1)N(\log N - 1)+ md \frac{N}{m} ( \log N - \log m - 1 ) \\
&- \sum_{ \bm r \in [m]^d} N m^{-d} \gamma_{\bm r}  ( \log N - d \log m + \log \gamma_{\bm r} - 1 ) + O(1).
\end{align*}
In particular, using the fact that $\sum_{\bm r \in [m]^d} m^{-d} \bm \gamma_{\bm r} = \int_{[0,1]^d} \gamma = 1$, after accounting for all the terms we find that
\begin{align} \label{eq:xXy}
\frac{1}{N} \log \mathbf{P}\left(
\bar{\gamma}^{\bm \sigma_N}_{\bm r}=\gamma_{\bm r}
\text{ for all }\bm r\in[m]^d
\right)  = - \int_{[0,1]^d} \gamma(s)\log \gamma(s)\mathrm{d}s + O(1/N).
\end{align}
The equation \eqref{eq:xXy} is the key idea in the proof of Theorem \ref{thm:KKRW} in the general $d \geq 2$. For the full technical details, see the appendix of \cite{KKRW}.

\section{Proof of Lemma \ref{lem:techa}} \label{sec:techa}

\begin{proof}[Proof of Lemma \ref{lem:techa}]
Without loss of generality, assume that $f$ takes values in $[0,1]$. In this case, for any measure $\nu$ on $E$ we can write
\begin{align*}
\int_E f(x) \nu(\mathrm{d}x) = \int_0^1 \nu(A_t)\mathrm{d}t,
\end{align*}
where $A_t := \{ x \in E : f(x) > t \}$. To establish the purported result, it is sufficient to show that $\mu_k(A_t) \to \mu(A_t)$ for almost all $t \in [0,1]$, since $0 \leq \mu_k(A_t) \leq 1$ for all $k,t$, and hence dominated convergence applies. In this direction, recall that we say that $A$ is a $\mu$-continuity set if $\mu(\partial A) = 0$, where $\partial A := \overline{A} \setminus \mathring{A}$.
The portmanteau theorem (see e.g.\ \cite[Theorem 5.25]{kallenberg}) states that if $(\mu_k)_{k \geq 1}$ converge weakly to $\mu$, then $\mu_k(A)$ converges to $\mu(A)$ for every $\mu$-continuity set $A$. Thus it remains to establish that $A_t$ is a $\mu$-continuity set for almost all $t \in [0,1]$.

We claim that $\partial A_t \subseteq \{ x \in E: f(x) = t \} \cup D_f$. To see this, suppose that $x \notin D_f$ and $f(x) \neq t$. Then since $f$ is continuous at $x$, there is a neighbourhood $U$ of $x$ such that either $f(y) > t$ for all $y \in U$, or $f(y) < t$ for all $y \in U$. In the first case, $U$ is contained in $A_t$, and in the second case, $U$ is contained in $\overline{A_t}^c$. Either way, $x$ is not an element of $\partial A_t := \overline{A_t} \setminus \mathring{A_t}$. It follows that $\partial A_t$ is a subset of $\{ x \in E: f(x) = t \} \cup D_f$. 

Now $\mu(D_f) = 0$. Moreover, $\mu( \{ x \in E : f(x) = t \} )$ may only be nonzero for at most countably many $t \in [0,1]$. In particular, since
\begin{align*}
\mu(\partial A_t) \leq \mu(D_f) + \mu( \{ x \in E: f(x) = t \} ),
\end{align*}
it follows that $\mu(\partial A_t)$ may only be nonzero for at most countably many $t \in [0,1]$. It follows from the portmanteau theorem that $\mu_k(A_t)$ converges to $\mu(A_t)$ for all but at most countably many $t \in [0,1]$, as required.
\end{proof}

\end{document}